\documentclass[10pt]{article}
\usepackage{graphicx} 
\usepackage{amsthm,amsfonts,amssymb,amsmath,oldgerm}
\usepackage{fullpage}
\usepackage{graphicx}
\usepackage{mathrsfs}
\usepackage{xcolor}
\numberwithin{equation}{section}
\usepackage{enumitem}
\usepackage{hyperref}
\usepackage{subcaption}

\makeatletter
\def\namedlabel#1#2{\begingroup
    #2%
    \def\@currentlabel{#2}%
    \phantomsection\label{#1}\endgroup
}
\makeatother

\newcommand{\R}{\mathbb R}
\newcommand{\C}{\mathbb C}

\newcommand{\ri}{\mathrm{i}}
\newcommand{\re}{\mathrm{e}}
\newcommand{\de}{\mathrm{d}}

\renewcommand{\Re}{\mathrm{Re}}
\renewcommand{\Im}{\mathrm{Im}}

\newcommand{\w}{\mathbf{w}}

\newtheorem{theorem}{Theorem}[section]
\newtheorem{proposition}[theorem]{Proposition}
\newtheorem{corollary}[theorem]{Corollary}
\newtheorem{lemma}[theorem]{Lemma}
\newtheorem{remark}[theorem]{Remark}
\theoremstyle{definition}

\allowdisplaybreaks[3]

\newcommand{\El}{\mathcal{L}}
\title{$L^\infty$-regularity thresholds for a generalized Gross-Pitaevskii equation: dispersive blow-up vs. nonlinear stability}

\author{Joannis Alexopoulos\thanks{Department of Mathematics, Karlsruhe Institute of Technology, Englerstra\ss e 2, 76131 Karlsruhe, Germany; \texttt{joannis.alexopoulos@kit.edu}}}

\begin{document}
\maketitle
\begin{abstract}
We study a generalized Gross-Pitaevskii equation arising in the mean-field description of exciton-polariton Bose-Einstein condensates. The dynamics close to macroscopic wave functions $r_0e^{i\omega_0 t +ik_0x}$ is governed by a stabilizing relaxation mechanism without suppressing dispersive blow-up. We show that this competition yields $L^\infty$-regularity thresholds. More precisely, we construct solutions for which arbitrarily small perturbations from $C^4_b$ can develop finite time blow-up in $C^3_b$ whereas  perturbations from $C^{4+\eta}_b$ for every  $\eta>0$, satisfy a global-in-time $C^3_b$-stability estimate.  The stability analysis is carried out in modulation spaces $M_{\infty,1}^k$.
\end{abstract}
\section{Introduction}
 We investigate the generalized Gross-Pitaevskii equation
 \begin{align}
 \label{GP_equ}
    \partial_t u = i\partial_x^2u + \gamma u + \beta |u|^2u, \quad (x,t) \in \R\times[0,\infty)
\end{align}
in the regime $\Re( \gamma) >0$,  $\Re (\beta) <0$ and $\Im(\beta)<0$.\footnote{In case of $\gamma =-\beta = i$, this is the classical Gross-Pitaevskii equation.}  The main physical motivation of the present paper is due to  Wouters and Curosotto  who build a bridge between the fields of nonlinear optics and quantum fluids in the context of exciton-polaritons in semiconductor cavities \cite{WCC} (with Ciuiti) and \cite{WC1, WC2}. Their theory explains  superfluidity properties  of polariton gases on the mean-field level and in an appropriate reduced description \cite{WC2}, the equation (\ref{GP_equ}) models the macroscopic wave functions $u$ of a polariton condensate.

 The balance between injection $\Re(\gamma)$ and dissipation $\Re(\beta)$ gives rise to the plane wave solutions $$u_p (x,t) =r_0 e^{i\omega_0t + ik_0 x}, \quad (x,t)\in \R \times [0,\infty)$$ when the dispersion relation 
\begin{align}
   i\omega_0 = -ik_0^2 +\gamma + \beta r_0^2  
\textrm{ equivalently }
\begin{split}
\begin{cases}
     r_0^2 =  -\frac{\Re(\gamma)}{\Re(\beta)}    \\\omega_0 + k_0^2 = \Im(\gamma) - \Im(\beta)\frac{\Re(\gamma)}{\Re(\beta)}
    \end{cases}
    \end{split}
\label{relation_dispersion}
\end{align}
is satisfied.  The central objective of
 this paper is to understand the role of initial perturbations close to the homogeneous condensates $u_p$ with non-trivial far-field behavior, i.e., perturbations which are nondecaying, nonconstant and nonperiodic as $x\rightarrow \pm \infty$.
 
 To state our main results, we introduce the spaces  
\begin{align*}
C_b^m(\R) = \{f \in C^{[m]}(\R,\C):  \partial_x^kf &\textrm{ is continuous and bounded for }  k =0,...,[m], \\
& \quad \textrm{ and } \partial_x^{[m]} f \textrm{ is } (m-[m])\textrm{-Hölder-continuous}\}, \quad m>0,    
\end{align*}
equipped with the norm
\begin{align*}
  ||f||_{C_b^m(\R)} = \sum_{k = 0}^{[m]} ||\partial_x^kf||_{L^\infty} + \sup_{x\neq y}|\frac{|\partial_x^{[m]}f(x) - \partial_x^{[m]}f(y)|}{|x-y|^{m-[m]}}|, \quad [m]  = \min\{n \in \mathbb{N}_0:n\leq m\}.
\end{align*}
Since $C_b^m(\R)$ contains every smooth bounded function for which all their derivatives are also bounded, we call $C_b^m(\R)$ a class of fully nonlocalized functions. On the one hand, nonlinear stability analyses against fully nonlocalized perturbations of periodic wave trains are established in parabolic situations \cite{bjorn1, alex3} and for a parabolic-hyperbolic system \cite{alex2}. Because of translational invariance, the spectrally stable configuration assumed in all these works is a quadratic touching at zero moving into the left half-plane - a so-called Goldstone mode - with remaining spectrum uniformly confined to the left half-plane. In the physical literature \cite{WC2}, it is well-known that the linearization about the wave function $u_p$ admits a Goldstone mode which can be parameterized as 
\begin{align*}
    \lambda_c(\xi) = -\frac{\Im(\beta)}{\Re(\beta)}\xi^2 + O(|\xi|^4), \quad |\xi|\ll 1,
\end{align*}
and indicates low-frequency relaxation close to $u_p$; we also refer to Figure \ref{fig_spec}.
On the other hand, the leading order dispersion of (\ref{GP_equ}) suggests an illposedness mechanism for nonlocalized data on the linear level: it is going back to observations of Hörmander, see e.g. \cite{arendt}, that the operator $i\partial_x^2$ generates a $C_0$-semigroup on $L^p$ if and only if $p = 2$.  Now, restricting to $L^\infty$-data, there even exist arbitrarily small $C^1_b$-data  which give rise to $L^\infty$-blow-up solutions for the linear propagator $e^{i\partial_x^2 t}$ due to the analysis of \cite{bps}. The authors apply this observation to show $L^\infty$-blow-up for various nonlinear dispersive equations such as nonlinear Schrödinger equations including the classical Gross-Pitaevskii equation \cite{bps} or the Korteweg-de-Vries equation  \cite{bonasaut}. In this spirit, we identify dispersive blow-up solutions starting close to $u_p$ in presence of the Goldstone mode. A key observation which still allow for dispersive blow-up is an identification of high-frequencies  to a propagator of the form $e^{\Re(\beta)r_0^2t}e^{i\partial_x^2 t}$. For this propagator, the  $C^1_b$-data from \cite{bps} still yield linear blow-up since the exponential damping $e^{\Re(\beta)r_0^2t}$ does not have a significant effect in finite times. In particular, we refined several parts of the previous analyses by identifying a blow-up rate and make assertions for higher derivatives. 
\begin{theorem}[Dispersive blow-up] 
    \label{theo_blow_up}
    Let $ \Re(\gamma)>0$, $\Re(\beta)<0$ and $\Im(\beta)<0$. Then, for all $k \in \mathbb{N}_0$ and $\varepsilon>0$,  there exist some $C>0$, $u_0 \in C_b^{k+1}(\R)\cap H^k(\R)$ with $||u_0||_{W^{k+1,\infty} \cap H^k} <\varepsilon$ and a solution $u \in u_p + L^\infty([0,T_0],C_b^{k}(\R))\cap C([0,\frac{1}{4}]; H^k(\R))$, $0<T_0<\frac{1}{4}$, to\footnote{The Duhamel formula should be spatially understood in $(H^k\cap C_b^k) \oplus H^1_{per}$.}
    \begin{align}
    \label{duhamel_theor_blow}
        u(t) = e^{i\partial_x^2  t}(u_0 + u_p(0)) + \int_0^t e^{i\partial_x^2(t-s)} (\gamma u(s) + \beta |u(s)|^2u(s)) \,ds
    \end{align}
      such that
    \begin{align*}
        ||\partial_x^k u(t)||_{L^\infty} \geq -C\log(\frac{1}{4} -t) \quad\textrm{ as } t\uparrow \frac{1}{4},
    \end{align*}
    and therefore
    \begin{align*}
        \lim_{t\uparrow \frac{1}{4}} ||u(t)||_{W^{k,\infty}} = \infty.
    \end{align*}
    \end{theorem}
\noindent This blow-up result together with the presence of the Goldstone mode opens an important question:
\begin{quote}
 Can more regular initial perturbations from  $C_b^m$, $m>k+1$, allow for $C_b^k$-stability estimates?    
\end{quote}
Our second main result answers this question positively and therefore leads to regularity threshold in $C_b^m$-spaces, see Corollary \ref{cor_stab}.  While $i\partial_x^2$ fails to define a   $C_0$-semigroup on $L^p$-spaces, an alternative is given by so-called modulation spaces $M_{p,1}^m$ for any $1\leq p\leq \infty$ on which $i\partial_x^2$ generates $C_0$-semigroups, see Appendix \ref{appendix_on_mod}. In particular, we  refer to \cite{klaus_mod_space} where it is shown that the defocusing NLS is globally well-posed in $M_{p,1}$ for every $1\leq p<\infty$ which captures the full localization range since $W^{2,p} \hookrightarrow M_{p,1} \hookrightarrow L^p$.  Very recently, this could be resolved for
$M_{\infty,1}$-data in  \cite{klaus2026}. The space $M_{\infty,1}$ is also known as Sjöstrand’s class.  As predicted in \cite{WC1}, the following theorem aligns with the physical prediction that the velocity $||\partial_x\phi(t) ||_{L^\infty}$ of a condensate of the form $u(t) = (\tilde{A}(t)+1)e^{i \phi(t)}u_p(t)$ decays diffusively, i.e., by  $(1+t)^{-\frac{1}{2}}$ as $t \rightarrow \infty$.

\begin{theorem}[Nonlinear stability in $M_{\infty,1}^3$]
\label{theo_stab}
Let $ \Re(\gamma)>0$, $\Re(\beta)<0$ and $\Im(\beta)<0$. There exist $C,\varepsilon>0$ such that for every initial datum $u_0 \in M_{\infty,}^3$ with $E_0 = ||u_0 - u_p(0)||_{M_{\infty,1}^3}< \varepsilon$ there exists a unique classical solution $u\in C([0,\infty),M_{\infty,1}^3(\R)\cap C^1([0,\infty),M_{\infty,1}^1(\R))$ to (\ref{GP_equ}) with  $u(0) = u_0$ such that
   \begin{align}
   \label{control_m_3}
       ||u(t)- u_p(t)||_{M_{\infty,1}^3} \leq C E_0
   \end{align}
   for all $t\geq 0$. Writing $u(t) = (\tilde{A}(t)+1)e^{i\phi(t)}u_p(t)$, we have the estimates
   \begin{align*}
       ||\tilde{A}(t)||_{L^\infty} \leq CE_0\log(2+t)(1+t)^{-1} ,\quad ||\partial_x^l( \phi(t)-h(t))||_{L^\infty} \leq CE_0(1+t)^{-\frac{l}{2}}
   \end{align*}
   with
   \begin{align*}
       h(t) = -\frac{b}{a}\log (e^{\frac{b}{a}\partial_x^2 t}e^{\begin{pmatrix}
           1 & -\frac{a}{b}\end{pmatrix} (A_0,\phi_0)^T}(\cdot-2k_0t))
   \end{align*}
   for $l = 0,1$ and all $t\geq 0$ where $a = \Re(\beta)r_0^2$ and $b = \Im(\beta)r_0^2$. For the definition of $(A_0,\phi_0)$, we refer to (\ref{identities}).   
\end{theorem}

The following corollary together with Theorem \ref{theo_blow_up} demonstrate that our nonlinear stability result is sharp in the sense that we have global existence for the solutions measured in $C^3_b$ for every small initial perturbation from $C^{4+\eta}_b$ while initial perturbations from $C^{4}_b$ lead to finite-time blow-up in $C^3_b$ .

\begin{corollary}
   [The $C_b^4$-threshold] 
   \label{cor_stab}
   Let $ \Re(\gamma)>0$, $\Re(\beta)<0$ and $\Im(\beta)<0$. We have:

 \vspace{0.1cm}
   \noindent
   \textbf{Nonlinear stability:}   For all $\eta>0$ there exist $C,\varepsilon>0$ such that for every initial datum $u_0 \in C_b^{4+\eta}(\R)$ with $E_0 = ||u_0 - u_p(0)||_{C^{4+ \eta}} < \varepsilon$ there exists a unique classical solution $u\in C([0,\infty),C_b^{3}(\R))\cap C^1([0,\infty),C_b^1(\R))$ to (\ref{GP_equ}) with  $u(0) = u_0$ such that
   \begin{align}
   \label{control_c_3}
       ||u(t)- u_p(t)||_{C^3_b} \leq C E_0
   \end{align}
   for all $t\geq 0$. 
   
   \vspace{0.1cm}
   \noindent
   \textbf{Blow-up:}
   For all $\varepsilon>0$ there exist $C>0$, $u_0 \in C_b^4(\R)$ with $||u_0-u_p(0)||_{C_b^4} < \varepsilon$ such that there exists a (mild) solution $u \in L^\infty([0,T_0],C_b^{4}(\R))$, $0<T_0<\frac{1}{4}$,  of (\ref{GP_equ}) such that
   \begin{align*}
       ||u(t)||_{C^3_b} \geq -C\log(\frac{1}{4}-t)
   \end{align*}
   as $t\uparrow\frac{1}{4}$.
\end{corollary}
 \begin{proof}
     The stability result follows immediately from Theorem \ref{theo_stab} and the embedding
 \begin{align*}
     C_b^{k+1+\eta} \hookrightarrow M^k_{\infty,1} \hookrightarrow  C^k_b, \quad \eta>0,
 \end{align*} 
 see Appendix \ref{appendix_on_mod}. The blow-up result is included in Theorem \ref{theo_blow_up}.
 \end{proof}

One advantage of nonlocalized data is that dilation leaves the $L^\infty$-norm invariant which allows in a natural manner to consider small dispersion regimes which gives a connection to the underlying spatially homogeneous equation, i.e., to the corresponding ordinary differential equation. 

\begin{corollary}[Dependence on the dispersion coefficient]
\label{cor_dispersion}
 Let $ \Re(\gamma)>0$, $\Re(\beta)<0$ and $\Im(\beta)<0$. There exist $C,\varepsilon>0$ such that for all $d>0$ and every $u_0 \in M_{\infty,}^3(\R)$ with $||u_0 - u_p(d^{-\frac{1}{2}}\cdot,0)||_{M_{\infty,1}^3(\R)} < \varepsilon$  there exists a unique classical  solution $u\in C([0,\infty),M_{\infty,1}^3(\R))\cap C^1([0,\infty),M_{\infty,1}^1(\R))$ to 
 \begin{align*}
     \partial_t u = di\partial_x^2u + \gamma u + \beta |u|^2u, \quad (x,t) \in \R\times[0,\infty),
 \end{align*}
 with  $u(0) = u_0$ such that
   \begin{align*}
       ||u(t)- u_p(d^{-\frac{1}{2}}\cdot,t)||_{L^\infty} \leq C \underbrace{\Bigl(
\sum_{i = 0}^3d^{\frac{i}{2}}||\partial_x^i(u_0 - u_p(d^{-\frac{1}{2}}\cdot,0))||_{M_{\infty,1}  } \Bigr) }_{=E_0}  \end{align*}
   for all $t\geq 0$. 
   Writing $u(t) = (\tilde{A}(t)+1)e^{i\phi(t)}u_p(t)$, we have the estimates
   \begin{align*}
       ||\tilde{A}(t)||_{L^\infty} \leq CE_0\log(2+t)(1+t)^{-1} ,\quad d^{\frac{l}{2}}||\partial_x^l( \phi(t)-h(t))||_{L^\infty} \leq CE_0(1+t)^{-\frac{l}{2}}
   \end{align*}
   with
   \begin{align*}
       h(t) = -\frac{b}{a}\log (e^{\frac{b}{a}\partial_x^2 t}e^{\begin{pmatrix}
           1 & -\frac{a}{b}\end{pmatrix} (\tilde{A}_0(d^\frac{1}{2}\cdot),\tilde{\phi}_0(d^{\frac{1}{2}}\cdot))^T}(\cdot-2k_0t))
   \end{align*}
   for $l = 0,1$ and all $t\geq 0$ where $a = \Re(\beta)r_0^2$ and $b = \Im(\beta)r_0^2$. In particular, if we consider spatially homogeneous plane waves $u_p(t) = r_0e^{i\omega_0 t}$, then we get  uniform bounds for small dispersions, i.e., 
   \begin{align*}
       \sup_{0<\tilde{d}\leq d}(||u(t)-u_p(t)||_{L^\infty} + \log(2+t)(1+t)||\tilde{A}(t)||_{L^\infty} + (1+t)^\frac{1}{2} ||\phi(t)||_{L^\infty} ) \leq C(||u_0-r_0||_{M_{\infty,1}} +d^\frac{1}{2})
   \end{align*}
   for all $t\geq 0$.
\end{corollary}
\begin{proof}
    The estimates follow by applying the Theorem \ref{theo_stab} to the function $\tilde{u}(x,t) = u(d^\frac{1}{2}x,t)$, the embedding $M^k_{\infty,1} \hookrightarrow  C^k_b$ and the scaling property from Lemma \ref{lemma_inequalites_on_mod} (iii).
\end{proof}

\begin{remark} 
Let $k\geq 4$. Choosing $u_0 \in C_b^{k+1+\eta}(\R)$ resp. $u_0 \in M_{\infty,1}^{k}(\R)$, one obtains control in (\ref{control_c_3}) resp. (\ref{control_m_3})  where $C^3_b$ resp. $ M_{\infty,1}^3(\R)$ is replaced by $C^k_b$ resp. $ M_{\infty,1}^{k}(\R)$. We will not go into detail on higher derivatives in the proof  for the nonlinear stability since it works the same modulo better decay rates for the derivatives due to the interpolation argument in Section \ref{section_proof_stab}.
\end{remark}

As we explain in Section \ref{section_outlook}, we expect that such a dichotomy between stability and blow-up is also present close to periodic Lugiato-Lefever waves by conceptionally following the analysis of this paper. Existing stability results deal with localized \cite{LLEperiod}, subharmonic \cite{uniform, LLE_ss} or sums of localized and periodic perturbations \cite{alex1} and it is unknown whether periodic Lugiato-Lefever waves are stable against fully nonlocalized perturbations.

\paragraph*{Notation.} 
\begin{itemize}
    \item  Let $S$ be a set, and let $A, B \colon S \to \R$. Throughout the thesis, the expression '$A(x) \lesssim B(x)$ for $x \in S$', means that there exists a constant $C>0$, independent of $x$, such that $A(x) \leq CB(x)$ holds for all $x \in S$. If the constant $C$ depends on some parameter $a$, then  we also write\ '$\lesssim_a$'. Alternatively, we also use the $O$-notation, meaning $A = O(B)$. Furthermore, when we write $A = O_{C_b^k}(B)$, we mean $||\frac{A}{B}||_{W^{k,\infty}} \lesssim 1$ for $A,B \in C_{b}^k$. 
    \item Let $w$ be a function, we write $\mathcal{F}w$ and $\hat{w}$ for the Fourier transform of $w$, and  $\mathcal{F}^{-1}w$ and $\check{w}$ for the inverse  Fourier transform of $w$.
\end{itemize}

\subsection{Outline and strategy of proof}

 We will consider two types of perturbation variable in Section \ref{section_pert_spec}: the multiplicative one and in polar coordinates. For both types, the linearization is given  by the operator
\begin{align*}
   \mathcal{L} =  \begin{pmatrix}
        0 & - \partial_x^2 \\
        \partial_x^2 & 0
    \end{pmatrix} + \begin{pmatrix}
        2a & 0 \\
        2b & 0
    \end{pmatrix}
\end{align*}
with $a = \Re(\beta)r_0^2$ and $b= \Im(\beta)r_0^2$. We use the Fourier transform to filtrate modes and refer to Figure \ref{fig_spec} for a qualitatively picture of the spectrum. This mode filtration allows to decompose $e^{\mathcal{L} t}$ as a sum of a damped dispersive propagator $e^{(i\partial_x^2 +a)t}$   and the heat semigroup $e^{\frac{b}{a} \partial_x^2 t}$ modulo a residual propagator which decays exponentially in time. This decomposition is established in Section \ref{sec_splitting_es_on_mod}.  Certain frequency decay of the multiplier for the residual propagator allows us to perform estimates from $M_{\infty,1}$ into $M_{\infty,1}$ and $L^2$ to $L^\infty$.  This shows that the residual propagator is not decisive neither for the dispersive blow-up nor the nonlinear stability. We also refer to \cite{GNT1, GNT2} establishing a similar decomposition of the semigroup for the classical Gross-Pitaevskii equation with application to dispersive blow-up in \cite{bps}. Corresponding proofs of multiplier theorems are delegated to Appendix \ref{appendix_mulitplier}. 

For the proof of the dispersive blow-up, the crucial observation in Section \ref{section_proof_blow} is that the damping term $e^{at}$ does not prevent blow-up through $e^{i\partial_x^2 t}$.  The diffusive propagator $e^{\frac{b}{a}\partial_x^2 t}$ does not contribute to the dispersive blow-up and also does not prevent it. For the proof of Theorem \ref{theo_stab} in Section \ref{section_proof_stab}, we use the special structure of the perturbation equations in polar coordinates. More precisely, the nonlinear phase evolution contributes only as spatial derivative which allows to isolate the critical dynamics in form of the viscous Hamilton-Jacobi equation $\partial_t \phi = \frac{b}{a} \phi_{xx} - \phi_x^2$. The connection of nonlinear Schrödinger equations to hydrodynamical structures comes into play by interpreting the perturbations in polar coordinates as variant of the Madelung transform. This analysis works entirely in the space $M_{\infty,1}^m$ and due to the isolation of the dispersive but exponentially damped propagator $e^{(i\partial_x^2 + a)t}$, every propagator is either exponentially damped or smoothing. Here, we may use that $$||e^{i\partial_x^2t}||_{M_{\infty,1}^m \rightarrow M_{\infty,1}^m} \lesssim (1+t)^\frac{1}{2}, \quad t\geq 0,$$ 
which means that the exponential damping $e^{at}$ dominates this growth.
One major issue is to control additional derivatives in the amplitude since nonlinear damping estimates used in \cite{alex3} or \cite{alex2} seem not available. This is circumvented by involving the multiplicative perturbation which satisfies a semilinear system. Using again that the propagators are either exponentially damped or smoothing and the equivalence of both types of perturbations, an interpolation argument in $M_{\infty,1}^m$ complements the iterative estimates on the Duhamel formulae. Here, the intuition on the regularity assumption $M_{\infty,1}^3$ comes from the fact that $|| \partial_x^3 e^{\partial_x^2 t}||_{M_{\infty,1}^3 \rightarrow M_{\infty,1}} \lesssim (1+t)^{-\frac{3}{2}}$ and that the right hand side is integrable in $t \in [0,\infty)$. This argument fails for $M_{\infty,1}^2$ since the interpolation estimates become too weak.

In Section \ref{section_outlook}, we discuss several possibly interesting directions for future research. This includes the possibility to accommodate large initial phases, higher dimensional variants, and adaptions to other systems.

\paragraph{Acknowledgement.} The author would like to thank Lars Eric Hientzsch and Björn de Rijk for their interest in this project, ongoing support and valuable discussions. In particular, the author thanks Björn de Rijk to suggest looking at an equation of the form (\ref{GP_equ}) when we talked about the author's fundamental idea of blow-up vs. stability results and Lars Eric Hientzsch on inspiring discussions on $L^\infty$-illposedness phenomena. The author also thanks Robert Wegner on valuable discussions on the physical background of (\ref{GP_equ}). This project is funded by the Deutsche Forschungsgemeinschaft (DFG, German Research Foundation) -- Project-ID 258734477 - SFB 1173 and partially Project-ID 491897824.

\section{Perturbation equations and spectral analysis}
\label{section_pert_spec}

\subsection{Formulation of perturbation equations}

We study two types of perturbative ansatzes, first the multiplicative perturbations and second the perturbations in polar coordinates. In particular, we will observe that we  may reduce ourselves to the constant state case by usual transformations which allows to simplify the equation without loss of generality.
We first introduce the multiplicative perturbation
  and make the  ansatz $u(x,t) = r(x,t) u_p(x,t)$ for the solution of (\ref{GP_equ}). We aim to study $r(x,t)$ close to the constant $1$.
  This leads to the equation
\begin{align*}
    \partial_t r = i\partial_x^2r - 2k_0\partial_xr - i(k_0^2+\omega_0)r + \gamma r+\beta r_0^2 \bar{r}r^2
\end{align*}
and setting $w = r-1$, we find the perturbation equation
\begin{align*}
    \partial_t w = i\partial_x^2w - 2k_0\partial_xw - i(k_0^2+\omega_0)w + \gamma w+\beta r_0^2 (2w+\bar{w}) + \beta r_0^2(w^2+2|w|^2 + |w|^2w).
\end{align*}
To study the spectrum of the linear part, we switch to real- and imaginary parts, use (\ref{relation_dispersion}), set $w_1 = \Re(w)$ and $w_2 = \Im(w)$ and arrive at 
\begin{align*}
   \partial_t \begin{pmatrix}
        w_1 \\w_2
    \end{pmatrix} = \begin{pmatrix}
        0 & -\partial_x^2 \\
        \partial_x^2 & 0
    \end{pmatrix}\begin{pmatrix}
        w_1 \\w_2
    \end{pmatrix} -2k_0 \begin{pmatrix}
        \partial_x & 0 \\
        0 & \partial_x
    \end{pmatrix}\begin{pmatrix}
        w_1 \\w_2
    \end{pmatrix}
     + M(a,b)\begin{pmatrix}
        w_1 \\w_2
    \end{pmatrix}
\end{align*}
with
\begin{align*}
M(a,b) &= \begin{pmatrix}
    \Re(\gamma) & - \Im(\gamma) + k_0^2 + \omega_0\\
    \Im(\gamma) - k_0^2 - \omega_0 & \Re(\gamma)
\end{pmatrix} + r_0^2\begin{pmatrix}
    3\Re(\beta) & -\Im(\beta) \\
    3 \Im(\beta) & \Re(\beta)
\end{pmatrix} \\
&= \begin{pmatrix}
    2a & -\Im(\gamma) + k_0^2 + \omega_0 - b \\
    \Im(\gamma)- k_0^2 - \omega_0 +3 b & 0
\end{pmatrix} \\
&= \begin{pmatrix}
    2a & 0 \\
    2b & 0
\end{pmatrix}
\end{align*}
where we set $$a = \Re(\beta) r_0^2 = -\Re(\gamma) \text{ and } b = \Im(\beta)r_0^2 = k_0^2 + \omega_0 - \Im(\gamma).$$ We set now $w = (w_1,w_2)^T$ and  shortly write 
\begin{align}
\label{pert_mult}
    \partial_t w = \El w + \mathcal{N}(w) 
\end{align}
with the linear operator
\begin{align*}
   \mathcal{L}(w) = \begin{pmatrix}
        0 & -\partial_x^2 \\
        \partial_x^2 & 0
    \end{pmatrix}\begin{pmatrix}
        w_1 \\w_2
    \end{pmatrix} + \begin{pmatrix}
         2a & 0 \\ 2b & 0
     \end{pmatrix}\begin{pmatrix}
        w_1 \\w_2
    \end{pmatrix} 
    \end{align*}
    and the nonlinearity
    \begin{align*}
    \mathcal{N}(w) =  \begin{pmatrix}
        a(3w_1^2 + w_2^2 + w_1^3 + w_2^2w_1) - b(2w_1w_2 + w_1^2w_2 + w_2^3) \\
        b(3w_1^2 + w_2^2 + w_1^3 + w_2^2w_1) + a(2w_1w_2 + w_1^2w_2 + w_2^3)
    \end{pmatrix}.
\end{align*}
We may reduce ourselves to the constant state: since $\Im(\gamma)$ disappears in the perturbation equations (\ref{pert_mult}), we may assume $\omega_0 = 0$ by setting $ \Im(\gamma) \mapsto \Im(\gamma) - \omega_0 $. Furthermore, by switching to the moving frame $ x\mapsto x-2k_0t$, we may assume $k_0 = 0$.
Now, we write the solution of (\ref{GP_equ}) in polar coordinates $u(x,t) = e^{A(x,t) + i \phi(x,t)}u_p(x,t)$ for $A,\phi \in \R$, $u_p \equiv r_0$. 
 Inserting the ansatz, we obtain
 \begin{align*}
     u_p(A_t + i\phi_t) = \partial_t u &= i \partial_x^2u + \gamma u + \beta |u|^2u  \\
     &= iu_p(A_{xx} + i \phi_{xx} + A_x^2- \phi_x^2 + iA_x\phi_x) + u_p\beta_0 r_0^2(e^{2A}-1).
 \end{align*}
 Comparing real and imaginary terms, we  arrive at
 \begin{align*}
     A_t &= - \phi_{xx} - 2A_x\phi_x + a(e^{2A}-1) \\
     \phi_t &= A_{xx} + A_x^2 - \phi_x^2 + b(e^{2A}-1)
 \end{align*}
 yielding 
\begin{align}
\label{pert_polar}
  \begin{pmatrix}
      A_t \\
      \phi_t
  \end{pmatrix} = \mathcal{L}\begin{pmatrix}
      A \\
      \phi
  \end{pmatrix}  +  \mathcal{N}_p(A,\phi) 
\end{align}
where we set
\begin{align*}
    \mathcal{N}_p(A,\phi) = \begin{pmatrix}
      -2\phi_xA_x + a(e^{2A} -1  -2A)\\
      b(e^{2A} -1  -2A) + A_x^2 - \phi_x^2
  \end{pmatrix}.
\end{align*}

\begin{proposition}[Local Existence in $w$]
\label{prop_ex_w}
   Let $k \in \mathbb{N}_{\geq 2}$ and $w_0 \in M_{\infty,1}^k$. Then there exists a maximal time $T_{max} \in (0,\infty]$ and  a unique classical solution $w \in C([0,T_{max}), M_{\infty,1}^k)\cap C^1([0,T_{max}), M_{\infty,1}^{k-2})$ to (\ref{pert_mult}). In particular, $T_{max}<\infty$ if and only if $\lim_{t\uparrow T_{max}}||w(t)||_{M_{\infty,1}^{k-2}} = \infty$. Moreover, $w$ fulfills the Duhamel formula
   \begin{align*}
       w(t) = e^{\mathcal{L}t}w_0 + \int_0^t e^{\mathcal{L}(t-s)}\mathcal{N}(w(s)) \,ds
   \end{align*}
   for all $t \in [0,T_{max})$.
\end{proposition}
\begin{proof}
   Note that $\mathcal{L}$ generates a strongly continuous semigroup on $M_{\infty,1}^k$ because of \cite[Theorem 1.3]{nagel} and $M_{\infty,1}^k$ is a Banach algebra. Therefore  the proof follows by standard semigroup theory, see e.g. \cite{CA98}.
\end{proof}
Writing
    \begin{align*}
        w_1 + iw_2 + 1 = e^{A + i\phi}
    \end{align*}
    provides the identities
    \begin{align}
    \begin{split}
        w_1 &= e^A \cos(\phi)- 1 \\
        w_2 &= e^A \sin(\phi)   ,      
    \end{split}
    \label{identities}
\quad
    \begin{split}
        A &= \log\left(\sqrt{(w_1+1)^2 + w_2^2}\right)\\
        \phi &= \arctan\left(\frac{w_2}{w_1+1}\right),
     \end{split}
     \end{align}
     which are all given by compositions of analytic expressions.\footnote{Here we use that $M_{\infty,1}^k$, $k \in \mathbb{N}_0$, are algebras and note that $f \in M_{\infty,1}$ does not necessarily imply $|f| \in M_{\infty,1}$. See also Appendix \ref{appendix_on_mod} for the crucial properties of modulation spaces relevant to this paper.}

\begin{corollary}[Local Existence in polar coordinates]
\label{cor_ex_polar}
   Let $k \in \mathbb{N}_{\geq 2}$ and $w_0 \in M_{\infty,1}^k$ and $T_{max}$ as in Proposition \ref{prop_ex_w}. Let $0<T<T_{max}$ and assume that $\sup_{t \in [0,T]} ||w(t)||_{L^\infty} \leq \frac{1}{2}$. There exists a unique classical solution $(A,\phi)^T \in C([0,T], M_{\infty,1}^k)\cap C^1([0,T], M_{\infty,1}^{k-2})$ to (\ref{pert_polar}). Moreover, $w$ fulfills the Duhamel formula
   \begin{align*}
       (A(t),\phi(t))^T = e^{\mathcal{L}t}(A_0,\phi_0)^T + \int_0^t e^{\mathcal{L}(t-s)}\mathcal{N}_p(A(s),\phi(s)) \,ds
   \end{align*}
   for all $t \in [0,T]$.
\end{corollary}
\begin{proof}
    This follows immediately by a successive application of (\ref{identities}).
\end{proof}

\begin{lemma}
\label{lemma_relate_pert}
Let $k \in \mathbb{N}_{\geq 2}$ and $w_0 \in M_{\infty,1}^k$. 
    Let $w$ and $T_{max}$  be given  as in Proposition \ref{prop_ex_w} and $(A,\phi)$ and $0<T< T_{max}$ be given by Corollary \ref{cor_ex_polar}. There exists some $T$-independent constant $C>0$ such that 
\begin{align*}
\sup_{0\leq t\leq T}||\partial_x^k(A(t),\phi(t))||_{M_{\infty,1}} \leq C \sup_{0\leq t\leq T}|| w||_{M_{\infty,1}^k},
\end{align*}
\begin{align*}
  \sup_{0\leq t\leq T} ||w(t)||_{M_{\infty,1}} \leq C \sup_{0\leq t\leq T}   ||(A(t),\phi(t))||_{M_{\infty,1}},
\end{align*}
and 
\begin{align*}
    ||(A_0,\phi_0)||_{M_{\infty,1}} \lesssim ||w_0||_{M_{\infty,1}} 
\end{align*}
    provided that 
    $||(A,\phi)||_{M_{\infty,1}}+||w_0||_{M_{\infty,1}} \leq \frac{1}{2}$.
\end{lemma}
\begin{proof}
   This follows by Proposition \ref{prop_ex_w}, Corollary \ref{cor_ex_polar} and (\ref{identities}). 
\end{proof}

\subsection{The relaxation mechanism and spectral stability}

We compute the spectrum of $\mathcal{L}$ where  the dispersive high-frequent and diffusive low-frequent effects are illustrated in Figure \ref{fig_spec} summarizing main aspects from Proposition \ref{prop_spec}. 

\begin{figure}[h]
    \centering
    \includegraphics[width=0.4\linewidth]{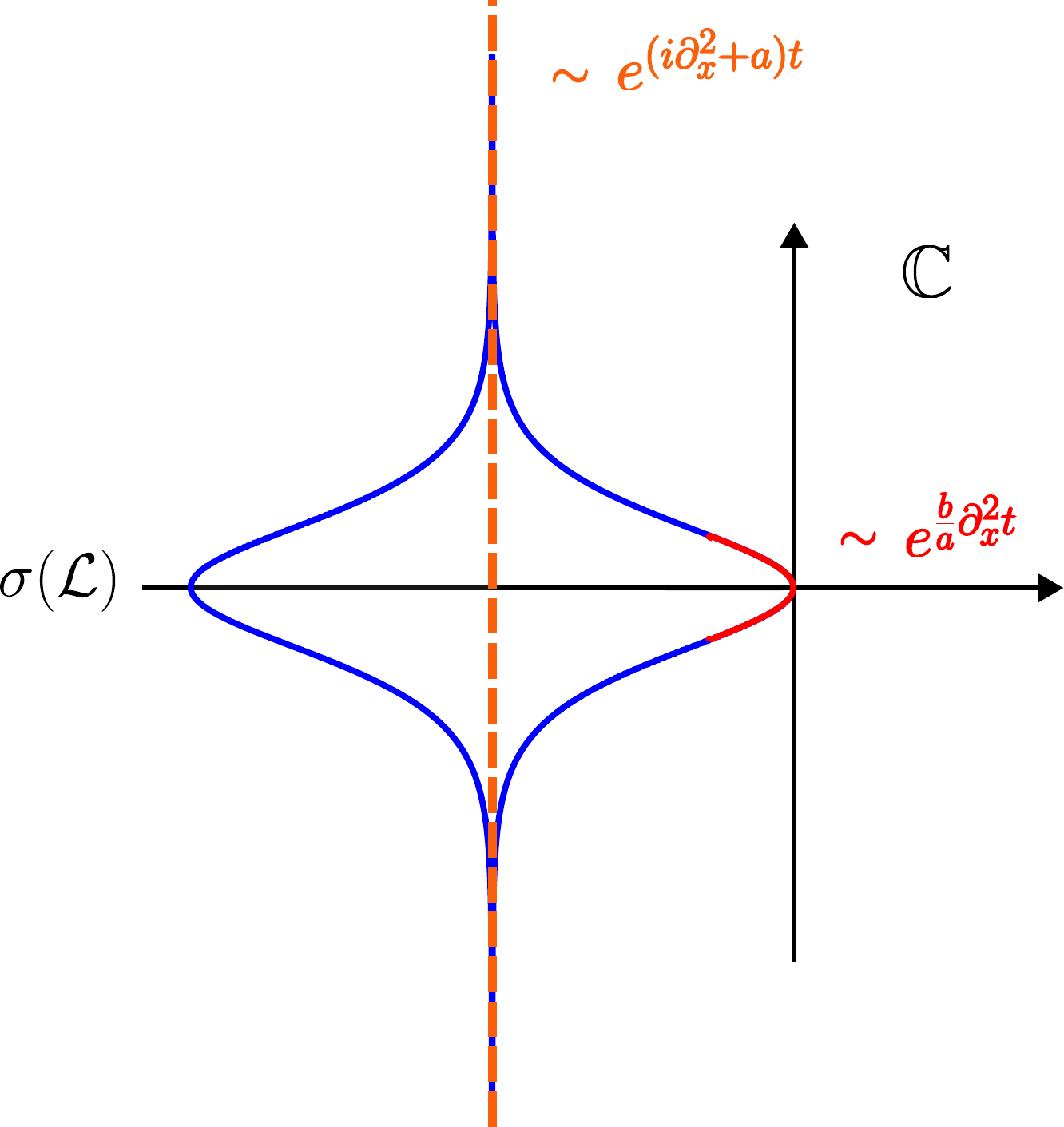}
    \caption{This picture qualitatively displays the spectrum computed in Proposition \ref{prop_spec}. A main ingredient of our analysis will be to isolate the high-frequency part and show that it coincides with the Fourier multiplier of a damped Schrödinger equation (in orange) modulo parts which yield exponentially decaying parts in time of the propagator. The red curve is a parabolic touching which we will identify to the heat equation. We note that this picture is qualitatively included in \cite[Fig.1 (b)]{WC1}	 where the authors call it the elementary excitation spectrum.}
    \label{fig_spec}
\end{figure}

\begin{proposition}[Spectrum of $\mathcal{L}$]
\label{prop_spec}
   Let $a,b<0$. The spectrum of $\mathcal{L}$ is given by
    \begin{align*}
        \sigma_{C_b}(\El) = \{\lambda_s(\xi): \xi \in \R\} \cup \{\lambda_c(\xi): \xi \in \R\} 
    \end{align*}
    where 
    \begin{align*}
       \lambda_s(\xi) =  a -  d(\xi^2), \quad \lambda_c(\xi) =  a + d(\xi^2), \quad \xi \in \R,
    \end{align*}
    and $d(\xi^2) = \sqrt{a^2+2b\xi^2-\xi^4}$ with the principal square root.
    In particular, there exist $\theta>0$ and $\xi_{l,1}, \xi_{h,1}>0$ such that
    \begin{align*}
       \sup_{\xi \in \R} \Re \lambda_s(\xi) \leq -\Theta, \quad 
       \sup_{|\xi|\geq\frac{\xi_{l,1}}{2}} \Re \lambda_c(\xi) \leq -\Theta,
    \end{align*}
    with the high-frequency expansions
    \begin{align*}
    \lambda_c(\xi) = a-i\xi^2 +ib +i\xi^{-2}f(\xi), \quad     \lambda_s(\xi) = \overline{\lambda_c(\xi)}=a+i\xi^2 -ib -i\xi^{-2}f(\xi), \quad |\xi|\geq \xi_{h,1}
    \end{align*}
    for some function $f \in C_b^2(\R)$ and the low-frequency expansion
    \begin{align*}
        \lambda_c(\xi) = - \frac{b}{a} \xi^2 + O_{C^4_b}(|\xi|^4), \quad |\xi|\leq \xi_{l,1}.
    \end{align*}
    Furthermore, denoting $L(\xi)$ the Fourier symbol of $\mathcal{L}$, we can diagonalize for $\xi^2 \neq \xi_0^2$ with  $\xi_0 :=\sqrt{b+ \sqrt{a^2+b^2}}$, that is,
    \begin{align*}
    S(\xi)^{-1}L(\xi) S(\xi) = \begin{pmatrix}
        \lambda_s(\xi) & 0 \\
        0 & \lambda_c(\xi)
    \end{pmatrix}
\end{align*}
whereas \begin{align*}
    S(\xi) = 
    \begin{pmatrix}
            \frac{-a+d(\xi^2)}{-2b+\xi^2} &   \frac{-a-d(\xi^2)}{-2b+\xi^2} \\
            1 & 1
    \end{pmatrix} , \quad  S(\xi)^{-1} = \frac{-2b+\xi^2}{2d(\xi^2)}\begin{pmatrix}
        1 &  \frac{a+d(\xi^2)}{-2b+\xi^2} \\
        -1 & \frac{-a+d(\xi^2)}{-2b+\xi^2} 
    \end{pmatrix}.
\end{align*}
\end{proposition}
\begin{proof}
It holds
\begin{align*}
    \sigma_{C_{b}}(\mathcal{L}) = \bigcup_{\xi \in \R} \sigma(L(\xi))
\end{align*}
where the Fourier symbol of $\mathcal{L}$ is given by
\begin{align*}
L(\xi) = \begin{pmatrix}
    2a  & \xi^2  \\
    -\xi^2 + 2b & 0
\end{pmatrix} .   
\end{align*}
We solve for 
\begin{align*}
\begin{split}
   0 &=  \det(L(\xi) - \lambda(\xi)) \\
   &= (\xi^2 -2b)\xi^2 - 2a(\lambda(\xi)) + (\lambda(\xi))^2 \\
   &= (\xi^2 -2b)\xi^2  + \lambda(\xi)(-2a) + \lambda(\xi)^2
   \end{split}
\qquad\text{ which gives }
\begin{split}
    \lambda(\xi) &= \frac{2a  \pm \sqrt{(-2a)^2 - 4((\xi^2 -2b)\xi^2 )}}{2} \\
    &= a  \pm \underbrace{\sqrt{a^2 - (\xi^2-2b)\xi^2}}_{=:d(\xi)}.
    \end{split}
\end{align*}
Now, we take $a<0$. We have the regimes
\begin{align*}
   \Re \left( \sqrt{a^2 - (\xi^2-2b)\xi^2} \right) = 0, \quad (\xi^2-b)\xi^2 \geq a^2, \quad b\leq0
\end{align*}
and
\begin{align*}
    \Re \left( \sqrt{a^2 - (\xi^2-2b)\xi^2}\right)   &< -a, \quad 0\leq (\xi^2-2b)\xi^2  < a^2,  \quad b\leq 0\\
    \Re \left( \sqrt{a^2 - (\xi^2-2b)\xi^2}\right)   &= -a, \quad \xi = 0, \quad b\leq 0\\
     \Re \left(\sqrt{a^2 - (\xi^2-2b)\xi^2}\right) &>-a,\quad 0<\xi^2 < 2b, \quad b>0.
\end{align*}
We set
\begin{align*}
  \lambda_s(\xi ) = a  -\sqrt{a^2 - (\xi^2-2b)\xi^2}, \quad \lambda_c(\xi ) = a  +\sqrt{a^2 - (\xi^2-2b)\xi^2}, \quad \xi \in \R.
\end{align*}
Choosing $b<0$, a Taylor expansion yields
\begin{align*}
\lambda_c(\xi) &= - \frac{b}{a}\xi^2 + O_{C^4_b}(|\xi|^4) 
\end{align*}
close to $\xi = 0$.  To get the expansion for large $\xi$, we write 
\begin{align*}
    \lambda_c(\xi) = a - \xi^2\sqrt{\frac{a^2}{\xi^4}  + \frac{2b}{\xi^2} -1}
\end{align*}
and make a Taylor expansion for the function $ z\mapsto\sqrt{a^2z^2  + 2bz -1}   $ at $z = 0$ .
The computations of $S$ and $S^{-1}$ are straightforward.
\end{proof}

\begin{remark}[The phenomenological relaxation]
  We consider the equation (\ref{GP_equ}) as dispersive variant of the complex Ginzburg-Landau equation which is used as phenomenological relaxation for the Gross-Pitaevskii equation \cite{relaxtrap}. In contrast, (\ref{GP_equ}) is not parabolic and the aim of this remark is to motivate, on a formal, nonlinear level, that the long-time behaviour about the coherent states $u_p$ is still conducted by diffusion. 
Following the phenomenological interpretation of the parabolic relaxation processes from \cite{relaxtrap}, we regroup terms from the right hand side in (\ref{GP_equ}) as
\begin{align*}
     i\partial_x^2u + \gamma u + \beta |u|^2 =\underbrace{i\left(\partial_x^2 + \Im(\gamma)  +  \Im(\beta)|u|^2  \right)u}_{\text{classical GP}} + \underbrace{(\Re(\gamma) +\Re(\beta)|u|^2)u}_{= R(u) \text{ (weak relaxation) }}
\end{align*}
We consider the specific situation $\frac{\Im(\gamma)}{\Im(\beta)} = \frac{\Re(\gamma)}{\Re(\beta)}$ and the coherent state $u_p \equiv r_0$ with trapping potential $\Im(\beta)r_0^2$ where the second term represents a  relaxation process towards $u_p$ since $r_0^2 =-\frac{\Re(\gamma)}{\Re(\beta)}$ and therefore $R(u_p) \equiv 0$. Considering the solution
\begin{align}
\label{sol_relaxation}
u(t) = r_0(1  + \delta(t)),
\end{align}
the perturbation $\delta(t)$ represents the coherent excitation. In particular,  the relaxation parameter $\frac{\Re(\beta)}{\Im(\beta)} >0$ is  precisely the inverse of the relaxation time for the perturbation $\delta(t)$.
On a formal level, we aim to make this observation for the perturbation variable $w(t) = \delta(t)$ satisfying (\ref{pert_mult}) and based on the principle $ w_1w_2, w_1^2, \partial_t w_1   \ll w_1 \ll w_2$ and $\partial_t w_2 \sim w_2^2$. First, we have
\begin{align*}
   h.o.t &= -\partial_x^2 w_2 + aw_2^2 - bw_2^3 \\
   \partial_t w_2 &= bw_2^2 + aw_2^3 +h.o.t.
\end{align*}
Combining the both equations yields the nonlinear heat equation
\begin{align*}
    \partial_t w_2 = \frac{b}{a} \partial_x^2 w_2 + \left(\frac{b^2}{a} + a\right) w_2^3 + h.o.t.
\end{align*}
Similarly, one obtains,
\begin{align*}
    \partial_t w_1 = \frac{a}{b} \partial_x^2w_1 +4a w_1 + c w_2^2+ h.o.t
\end{align*}
for some $c \in \R$.
These two equations motivate the stability condition $a,b<0$ which gives diffusive decay in the first leading order equation and exponential slaving for the second one. For both equations one may use parabolic smoothing to control derivatives in $C^k_b$.  In this formal reduction, we suppressed dispersive contributions and the  dispersive blow-up seems not present. However, these dispersive effects obstruct parabolic smoothing properties in fully nonlocalized spaces, as well as known energy-type arguments as used in the hyperbolic-parabolic situation \cite{alex2}. We will show that the dispersive contributions are exponentially damped, which is the key observation to show nonlinear stability.   
\end{remark}

\section{Linear decomposition and estimates from \texorpdfstring{$M_{\infty,1}$}{Lg} to \texorpdfstring{$M_{\infty,1}$}{Lg}}
\label{sec_splitting_es_on_mod}
The spectral projections onto $\lambda_s(\xi)$ resp. $\lambda_c(\xi)$   are given by
\begin{align*}
    P_s (\xi) = S(\xi) \begin{pmatrix}
        1 & 0 \\
        0 & 0
    \end{pmatrix} S(\xi)^{-1} = \frac{-2b+\xi^2}{2d(\xi^2)}\begin{pmatrix}
          \frac{-a+d(\xi^2)}{-2b+\xi^2}  & -\frac{\xi^2}{-2b+\xi^2} \\
          1 & \frac{a+d(\xi^2)}{-2b+\xi^2}
    \end{pmatrix}
\end{align*}
and 
\begin{align*}
    P_c(\xi) = 1- P_s(\xi) =  \frac{-2b+\xi^2}{2d(\xi^2)} \begin{pmatrix}
         \frac{a+d(\xi^2)}{-2b+\xi^2}  & \frac{\xi^2}{-2b+\xi^2} \\
          -1 &    \frac{-a+d(\xi^2)}{-2b+\xi^2}
    \end{pmatrix}.
\end{align*}
We establish the expansions of the projectors for high- and low-frequencies.
\begin{proposition}
\label{lemma_proj_expansion}
There exist $0<\xi_{l,2}< \xi_0<\xi_{h,2}$ and $M_1,M_2 \in C_b^4([0,\xi_{l,2}];\R^2)$ such that
\begin{align*}
    P_c(\xi) = \begin{pmatrix}
       M_1(\xi)^T \\
       M_2(\xi)^T
    \end{pmatrix}, \quad M_2(\xi)^T = \begin{pmatrix}
        -\frac{b}{a} & 1
    \end{pmatrix} + O_{C^4_b}(\xi^2), \quad  M_1(\xi) =  O_{C^4_b}(\xi^2)\quad |\xi| \leq \xi_{l,2}
\end{align*}
and 
\begin{align*}
    P_c(\xi) = \frac{1}{2}\begin{pmatrix}
        1 & i \\
        -i & 1
    \end{pmatrix} +  \frac{1}{\xi^2}M_3(\xi), \quad  P_s(\xi) = \frac{1}{2}\begin{pmatrix}
        1 & -i \\
        i & 1
    \end{pmatrix} +  \frac{1}{\xi^2}M_4(\xi), \quad |\xi|\geq \xi_{h,2}
\end{align*}
for some $M_3,M_4 \in C_b^2([\xi_{h,2},\infty);\C^{2\times2})$. For the sake of readability, we write
\begin{align*}
  C_c=   \frac{1}{2}\begin{pmatrix}
        1 & i \\
        -i & 1
    \end{pmatrix}, \quad C_s = \bar{C_c}= \frac{1}{2}\begin{pmatrix}
        1 & -i \\
        i & 1
    \end{pmatrix}.
\end{align*}
\end{proposition}
\begin{proof}
For the expansion of $P_c(\xi)$ for low-frequencies, we write
\begin{align*}
    P_c(\xi) = \begin{pmatrix}
       \frac{a+ d(\xi^2)}{2d(\xi^2)} & \frac{\xi^2}{2d(\xi^2)} \\
       \frac{-\xi^2}{2d(\xi^2)} & 0
    \end{pmatrix} + \begin{pmatrix}
       0 & 0\\
       \frac{2b}{2d(\xi^2)} & \frac{-a+d(\xi^2)}{2d(\xi^2)} 
    \end{pmatrix}:= P_c^1(\xi) + P_c^2(\xi).
\end{align*}
With the help of an Taylor expansion, it is straightforward to see that $P_c^1(\xi)$ and $P_c^2(\xi) - \begin{pmatrix}
    0 & 0 \\
    \frac{2b}{2a} & 1
\end{pmatrix}$ are $O_{C^4_b}(\xi^2)$. 
For the high-frequency part, one writes  
\begin{align*}
   \frac{\xi^2}{2d(\xi^2)} = \frac{1}{2} \frac{1}{\frac{a}{\xi^2} - \sqrt{\frac{a^2}{\xi^4}  + \frac{2b}{\xi^2} -1}}
\end{align*}
and Taylor expand this expression in $\frac{1}{\xi^2}$.
\end{proof}
We invoke the Fourier inversion formula 
\begin{align*}
    e^{\El t}v(x) = \frac{1}{2\pi}\lim_{R \rightarrow\infty} \int_{-R}^R e^{i\xi x} e^{L(\xi )t} \hat{v}(\xi)  \,d\xi
\end{align*}
to represent the semigroup $e^{\El t}$.
Let $\xi_{l,i}$, $\xi_{h,i}$, $i =1,2$, be as in Propositions \ref{prop_spec} and \ref{lemma_proj_expansion}, we set $\xi_l = \frac{1}{2}\min\{\xi_{l,1},\xi_{l,2}\}$ and $\xi_h = 2\max\{\xi_{h,1},\xi_{h,2}\}$ and introduce a smooth temporal cut-off function $\chi \colon [0,\infty) \to \R$ satisfying $\chi(t) = 0$ for $t \in [0,1]$ and $\chi(t) = 1$ for $t \in [2,\infty)$ and let $\rho_l \colon \R \to \R$ be a smooth frequency cut-off function, which is supported on $(-\xi_l,\xi_l)$ and satisfies $\rho_l(\xi) = 1$ for $\xi \in \smash{[-\frac12 \xi_l, \frac12 \xi_l]}$. Now, we establish the decomposition
\begin{align*}
    e^{\El t}v(x) = \chi(t)\lim_{R\rightarrow\infty} S_c^R(t) v(x) + \chi(t)\lim_{R\rightarrow\infty}S_e^R(t)v(x) +  (1-\chi(t))e^{\mathcal{L}t}v(x) 
\end{align*}
 with 
 \begin{align}
      S_c^R(t) v(x) = \frac{1}{2\pi}\int_{-R}^R  \rho_l(\xi)e^{i\xi x} e^{L(\xi )t} P_c(\xi)\hat{v}(\xi)  \,d\xi, \quad S_c(t) v=  \lim_{R\rightarrow \infty} S_c^R(t)v
      \label{s_c_fix}
 \end{align}
 Introducing another frequency cut-off function $\rho_h \colon \R \to \R$, which is supported on $[-\xi_h,\xi_h]^c$ and satisfies $\rho_h(\xi) = 1$ for $\xi \in \smash{(-2\xi_h, 2 \xi_h)^c}$, we split
 \begin{align*}
     S_e^R(t)v = \sum_{i=1}^3 S_{e,i}^R(t)v
 \end{align*}
 with 
 \begin{align}
 \begin{split}
    S_{e,1}^R(t)v(x) &= \frac{1}{2\pi}\int_{-R}^R  \rho_h(\xi)e^{i\xi x} e^{L(\xi )t} \hat{v}(\xi)  \,d\xi \\
     S_{e,2}^R(t)v(x) &= \frac{1}{2\pi}\int_{-R}^R  \rho_l(\xi)e^{i\xi x} e^{L(\xi )t} P_s(\xi)\hat{v}(\xi)  \,d\xi \\
     S_{e,3}^R(t)v(x) &= \frac{1}{2\pi}\int_{-R}^R  (1-\rho_l(\xi))(1-\rho_h(\xi))e^{i\xi x} e^{L(\xi )t} \hat{v}(\xi)  \,d\xi 
      \end{split}
      \label{s_e_fix}
 \end{align}
   noting that $\xi_l < \xi_h$. We set $S_{e,i}(t) = \lim_{R\rightarrow \infty} S_{e,i}^R(t)$, $i = 1,2,3$,  $S_{e,4}(t) =(1-\chi(t))e^{\mathcal{L}t} $ and $S_e(t) = \chi(t)\Sigma_{i = 1}^3 S_{e,i}(t) + S_{e,4}(t)$.

\subsection{Isolation of the dispersive and damped high-frequency components} 

The purpose of this section is to investigate $S_e(t)$ from $M_{\infty,1}$ into $M_{\infty,1}$. We first establish the estimates on $S_{e,2}(t)$ and $S_{e,3}(t)$. Recall that
\begin{align}
\label{split_linear_in_modes_diagonalized}
    e^{L(\xi)t} = 
        e^{\lambda_s(\xi)t}   P_s(\xi ) +  e^{\lambda_c(\xi)t} 
     P_c(\xi )
\end{align}
for $|\xi|\neq \xi_0$.

\begin{proposition}
  There exists some $\delta_1>0$ such that It holds
    \begin{align*}
        ||\lim_{R\rightarrow \infty} S_{e,2}^R(t)||_{M_{\infty,1} \rightarrow M_{\infty,1}} + ||\lim_{R\rightarrow \infty} S_{e,3}^R(t)||_{M_{\infty,1} \rightarrow M_{\infty,1}}+ ||S_{e,4}(t)||_{M_{\infty,1} \rightarrow M_{\infty,1}} \lesssim e^{-\delta_1t}
    \end{align*}
    for all $t\geq 0$.
\end{proposition}
\begin{proof}
    The estimate on $S_{e,4}(t)$ is clear thanks to Lemma \ref{lemma_bounds}. We observe
    \begin{align*}
        |e^{L(\xi_0)t}| \lesssim e^{-\frac{\Theta}{2}t}. 
    \end{align*}
    By the inequality
    \begin{align*}
        |e^{A+B}|\leq e^{|A|}e^{|B|}, \quad A,B \in \C^n
    \end{align*}
    and the continuous dependence of $L(\xi)$ in $\xi$, we find
    \begin{align*}
        \sup_{\xi \in (\xi_0-\varepsilon    ,\xi_0 +\varepsilon)}|e^{L(\xi)t}|\lesssim e^{-\frac{\Theta}{4}t}
    \end{align*}
    for some small $\varepsilon>0$. Now, by (\ref{split_linear_in_modes_diagonalized}), it is clear that
    \begin{align*}
        \sup_{\xi \in [\frac{\xi_{l}}{2}, 2\xi_h] \setminus (\xi_0-\varepsilon    ,\xi_0 +\varepsilon)} |e^{L(\xi)t}| \lesssim e^{-\Theta t}.
    \end{align*}
    Since $\partial_\xi e^{L(\xi) t} = L(\xi) t e^{L(\xi)t}$ and also for higher derivatives, we conclude that 
    \begin{align*}
        ||(1-\rho_l(\xi))(1-\rho_h(\xi))e^{L(\xi )t}||_{M_{\infty,1}} \lesssim  ||(1-\rho_l(\xi))(1-\rho_h(\xi))e^{L(\xi )t}||_{C^2_b} \lesssim e^{-\frac{\Theta}{8}t}
    \end{align*}
    and therefore the estimate for $S_{e,3}(t)$ as consequence of Lemma \ref{lemma_groe}. Since also
     \begin{align*}
        \sup_{\xi \in [-\xi_l,\xi_l]} ||e^{L(\xi)t} P_s(\xi)|| \lesssim e^{-\Theta t},
    \end{align*}
    the estimate on $S_{e,2}^R(t)$ follows.
\end{proof}
Let $|\xi| \geq \xi_h$. We start to estimate $S_{e,1}^R(t)$ by invoking (\ref{split_linear_in_modes_diagonalized}) and recall from Propositions \ref{prop_spec} and \ref{lemma_proj_expansion} that
\begin{align*}
    P_c(\xi) = C_c + M_3(\xi) \textrm{ and }
    \lambda_c(\xi) = a +ib- i\xi^2 + i\xi^{-2}f(\xi).
\end{align*}
Setting
\begin{align*}
    e^{\lambda_c(\xi)t} = e^{(a-ib)t}(r_c(\xi,t) + e^{i\xi^2t}),
\end{align*}
we detect
\begin{align}
\begin{split}
    &|r_c(\xi,t)|  = |e^{-i\xi^2t}- e^{-i(\xi^2 + \frac{1}{\xi^2}f(\xi))t}|\lesssim t\xi^{-2}, \\
    &|\partial_\xi r_c(\xi,t)| \lesssim  |2i\xi t(e^{-i\xi^2t}- e^{-i(\xi^2 + \frac{1}{\xi^2}f(\xi))t})| + \frac{t}{|\xi|}\lesssim (t+ t^2) \xi^{-1} \\
    &|\partial_\xi^2 r_c(\xi,t)|\lesssim   |(2i\xi t)^2 (e^{-i\xi^2t}- e^{-i(\xi^2 + \frac{1}{\xi^2}f(\xi))t})| + t + t^2 \lesssim t+t^2+ t^3, \quad t\geq 0, \quad 
    \end{split}
    \label{high_C_2_r_estimates}
\end{align}
In particular, 
\begin{align*}
    ||r_c(\xi,t)||_{M_{\infty,1}} \lesssim t^3, \quad t\geq 1.
\end{align*}
Furthermore,
\begin{align}
\label{high_C_2_M_estimates}
\begin{split}
    ||e^{\lambda_c(\xi)t}\xi^{-2}M_3(\xi)||_{L^\infty} &\lesssim e^{at} |\xi|^{-2} \\
||\partial_\xi(e^{\lambda_c(\xi)t}\xi^{-2}M_3(\xi))||_{L^\infty}  &\lesssim e^{at} (1+t)|\xi|^{-1}
\\
||\partial_\xi^2(e^{\lambda_c(\xi)t}\xi^{-2}M_3(\xi))||_{L^\infty}  &\lesssim e^{at}(1+t + t^2), \quad t\geq 0,
\end{split}
\end{align}
and therefore
\begin{align*}
||e^{\lambda_c(\xi)t}\xi^{-2}M_3(\xi)||_{M_{\infty,1}} \lesssim e^{\frac{a}{2}t}, \quad t\geq 1.
\end{align*}
Similarly, one argues for $\lambda_s$ with the corresponding $r_s$ and $e^{\lambda_s(\xi)t}(\xi^{-2})M_4(\xi)$. 
Now, we rewrite
\begin{align*}
   S_{e,1}^R(t) =  \int_{-R}^R \rho_h(\xi)e^{i\xi x} e^{L(\xi )t} \hat{v}(\xi)  \,d\xi = \int_{-R}^R \rho_h(\xi)e^{i\xi x}I_h(\xi,t) \hat{v}(\xi) \,d\xi
\end{align*}
whereas
\begin{align*}
    I_h(\xi,t) = e^{\lambda_c(\xi)t}\xi^{-2}M_3(\xi) +e^{\lambda_s(\xi)t}\xi^{-2}M_4(\xi) + 
        e^{\lambda_s(\xi)t}   C_s + 
        e^{\lambda_c(\xi)t}   C_c 
\end{align*}
and
\begin{align*}
        e^{\lambda_s(\xi)t}   C_s = e^{(a-ib)t}  \left( 
        e^{i\xi^2t}   + 
        r_s(\xi,t)   \right)C_s, \quad
        e^{\lambda_c(\xi)t}   C_c = e^{(a+ib)t}  \left( 
        e^{-i\xi^2t}   + 
        r_c(\xi,t)  \right)C_c.
\end{align*}
Collecting residual terms in
\begin{align}
\begin{split}
    r(\xi,t) =  &\rho_h(\xi)\left(e^{(a+ib)t}     
        r_c(\xi,t) C_c + e^{(a-ib)t}     
        r_s(\xi,t)  C_s \right) \\
    &- (1-\rho_h(\xi))\left(e^{(a+ib)t}     
       e^{-i\xi^2t}  C_c + e^{(a-ib)t}     
       e^{i\xi^2t}C_s\right) \\
       &+\rho_h(\xi)\left(e^{\lambda_c(\xi)t}\xi^{-2}M_3(\xi) +e^{\lambda_s(\xi)t}\xi^{-2}M_4(\xi)\right),
       \end{split}
       \label{r_fix}
\end{align}
we summarize
\begin{align*}
    \int_{-R}^R e^{i\xi x}I_h(\xi,t) \hat{v}(\xi) \,d\xi = \int_{-R}^R e^{i\xi x} r(\xi,t) \hat{v}(\xi) \,d\xi +  (S_{\textrm{D}}^R(t) v)(x),
\end{align*}
with 
\begin{align*}
   (S_{\textrm{D}}^R(t)v)(x)=  e^{(a+ib)t}\int_{-R}^R e^{i\xi x} e^{-i\xi^2 t}C_c\hat{v}(\xi) \,d\xi 
    + e^{(a-ib)t}\int_{-R}^R e^{i\xi x} 
       e^{i\xi^2t}  C_s\hat{v}(\xi) \,d\xi.
\end{align*}
By Lemma \ref{lemma_groe}, we have
\begin{align*}
    ||\lim_{R\rightarrow \infty}\int_{-R}^R e^{i\xi x} r(\xi,t) \hat{v}(\xi) \,d\xi||_{M_{\infty,1}} \lesssim e^{\frac{a}{2}t} ||v||_{M_{\infty,1}}.
\end{align*}
Furthermore, setting 
\begin{align*}
    S_\textrm{D}(t) = \frac{1}{2\pi}\lim_{R\rightarrow \infty} S_\textrm{D}^R(t),
\end{align*}
we obtain
\begin{align}
\label{def_s_d}
   S_\textrm{D}(t)v &=   e^{(a+ib)t}e^{i\partial_x^2 t} C_c v +e^{(a-ib)t}e^{-i\partial_x^2 t} C_s v.
\end{align}
Due to Lemma \ref{lemma_bounds}, we bound
\begin{align*}
    ||e^{\pm i \partial_x^2 t} v||_{M_{\infty,1}}  \lesssim (1+t)^\frac{1}{2} ||v||_{M_{\infty,1}}
\end{align*}
and conclude
\begin{align*}
||S_\textrm{D}(t) v||_{M_{\infty,1}} \lesssim e^{\frac{a}{2}t}||v||_{M_{\infty,1}}.
\end{align*}

We summarize the just shown estimates in the following.

\begin{proposition}
\label{prop_se_damped_stab}
There exists some $\delta_2>0$ such that
\begin{align*}
    ||S_e(t)||_{M_{\infty,1} \rightarrow M_{\infty,1} } \lesssim e^{-\delta_2 t}
\end{align*}
for all $t\geq 0$.
\end{proposition}
A crucial observation is the following: if $v = \frac{1}{2}(w+ \bar{w},\frac{1}{i}(w-\bar{w}))$ for some $w \in \C$, then we can rewrite
\begin{align}
\begin{split}
S_{\textrm{D}}(t) v
&= \frac{1}{4} e^{at}\left(e^{i(b+\partial_x^2)t}\begin{pmatrix}
    1 & i \\
    -i & 1
\end{pmatrix} \begin{pmatrix}
   w+ \bar{w}\\ \frac{1}{i}(w-\bar{w}) 
\end{pmatrix}  + \overline{e^{i(b+\partial_x^2)t}\begin{pmatrix}
    1 & i \\
    -i & 1
\end{pmatrix} \begin{pmatrix}
   w+ \bar{w} \\ \frac{1}{i}(w-\bar{w}) 
\end{pmatrix}}\right)
\\&= e^{at}\Re\left(e^{ibt}e^{i\partial_x^2t} \begin{pmatrix}
    w\\-iw
\end{pmatrix}\right) \\
&=e^{at}\begin{pmatrix}
    \Re(e^{i(\partial_x^2+b)t}w) \\
    \Im(e^{(i(\partial_x^2+b)t}w)
\end{pmatrix}.
\end{split}
\label{S_D_rewrite}
\end{align}
\subsection{Low-frequency decomposition and estimates}
We decompose the low-frequency  propagator
\begin{align*}
   \chi(t) S_{c}(t) v(x) =  \frac{\chi(t)}{2\pi}\lim_{R \rightarrow \infty}\int_\R\int_{-R}^R  \rho_l(\xi)e^{i\xi (x-y)}  e^{\lambda_c(\xi)t}
    P_c(\xi) \,d\xi v(y) \,dy = \begin{pmatrix}
         \tilde{S}_A(t)v(x) \\ (S_p(t) + \tilde{S}_r(t))v(x)
    \end{pmatrix}
\end{align*}
whereas
\begin{align*}
 \tilde{S}_A(t)v(x) &= \lim_{R \rightarrow \infty} \left(\frac{\chi(t)}{2\pi}\int_{-R}^R  \rho_l(\xi)e^{i\xi (x-y)} e^{\lambda_c(\xi)t}  M_1(\xi)^T \,d\xi v(y)\,dy\right), \\ 
\tilde{S}_r(t)v(x) &= \lim_{R \rightarrow \infty} \Bigl(\frac{\chi(t)}{2\pi}\int_{-R}^R  \rho_l(\xi)e^{i\xi (x-y)} e^{\lambda_c(\xi)t} (M_2(\xi)^T-\begin{pmatrix}
        -\frac{b}{a} & 1
    \end{pmatrix}) \,d\xi v(y)\,dy \\ 
    &\quad  + \frac{\chi(t)}{2\pi}\int_{-R}^R  \rho_l(\xi)e^{i\xi (x-y)} (e^{\lambda_c(\xi)t} - e^{\lambda_p(\xi)t}) \begin{pmatrix}
        -\frac{b}{a} & 1
    \end{pmatrix}\,d\xi v(y)\,dy \\
    &\quad  - \frac{\chi(t)}{2\pi}\int_{-R}^R  (1-\rho_l(\xi))e^{i\xi (x-y)} e^{\lambda_p(\xi)t} \begin{pmatrix}
        -\frac{b}{a} & 1
    \end{pmatrix}\,d\xi v(y)\,dy \Bigr)\\
    &\quad- (1-\chi(t)) e^{\frac{b}{a}\partial_x^2 t},
    \\
    S_p(t)v(x) &= (e^{\frac{b}{a}\partial_x^2 t} \begin{pmatrix}
        -\frac{b}{a} & 1
    \end{pmatrix}v)(x)
\end{align*}
and $\lambda_p(\xi)  = -\frac{b}{a} \xi^2$. 

\begin{proposition}
\label{prop_low_fre_estimates}
There exists some $C>0$ such that
 \begin{align*}
     ||\partial_x^k \tilde{S}_A(t)||_{M_{\infty,1} \rightarrow M_{\infty,1}} \leq C (1+t)^{-1 - \frac{k}{2}},
 \end{align*}
 and
\begin{align*}
 ||\partial_x^k \tilde{S}_r(t)||_{M_{\infty,1} \rightarrow M_{\infty,1}} \leq C t^{-\frac{k}{2}}(1+t)^{-1} , \quad ||\partial_x \tilde{S}_r(t)||_{M_{\infty,1}^1 \rightarrow M_{\infty,1}} \leq C (1+t)^{-1},
\end{align*}
\begin{align*}
    ||\partial_x^k S_p(t)||_{M_{\infty,1} \rightarrow M_{\infty,1}} \leq C t^{-\frac{k}{2}}, \quad ||\partial_x S_p(t)||_{M_{\infty,1}^1 \rightarrow M_{\infty,1}} \leq C (1+t)^{-\frac{1}{2}}
\end{align*}
for $k = 0,1$, and
\begin{align*}
    ||\partial_x^l \chi(t)S_c(t)||_{M_{\infty,1} \rightarrow M_{\infty,1}} \leq C(1+t)^{-\frac{l}{2}},
\end{align*}
for $l = 0,1,2,3$.
\end{proposition}
\begin{proof}
 The estimate on $S_p(t)$  and the one on the last term in  $\tilde{S_r}(t)$ follow immediately from Lemma \ref{lemma_inequalites_on_mod} (i). The estimates on $\tilde{S}_A(t)$ and the first two integrals in $\tilde{S}_r(t)$ follow from Lemma \ref{lemma_low_frequency} since $|M_1(\xi)| \lesssim \xi^2$, $|M_2(\xi)^T - \begin{pmatrix}
     -\frac{b}{a} &1
 \end{pmatrix}| \lesssim \xi^2$ and $$ |e^{\lambda_p(\xi)t}(1 - e^{(\lambda_c(\xi)-\lambda_p(\xi))t})| \lesssim  e^{-d\xi^2t}\xi^4t$$
 for some $d>0$, see also \cite{bjorn1, björn2}. The third integral from $\tilde{S}_r(t)$ can be immediately estimated using Lemma \ref{lemma_high-frqu_para}.
 \end{proof}
\begin{corollary}
\label{cor_interchange_der}
    It holds
    \begin{align*}
 \partial_x^l  S_e(t) = S_e(t) \partial_x^l \textrm{ and }     \partial_x^l e^{Lt} = S_e(t) \partial_x^l + \chi(t)\partial_x^lS_c(t)
    \end{align*}
    for all $l \in \mathbb{N}_0$.
    \end{corollary}
\section{Nonlinear stability analysis}
\label{section_proof_stab}
\subsection{Nonlinear decomposition}
 Using the decomposition of the linear evolution $e^{\El t}$, we set
\begin{align*}
    S_A(t) = \tilde{S}_A(t)+\begin{pmatrix}
        1 & 0
    \end{pmatrix} S_e(t) \textrm{ and }
    S_r(t) = \tilde{S}_r(t) + \begin{pmatrix}
        0 & 1
    \end{pmatrix} S_e(t).
\end{align*}
Let  $w_0 \in M_{\infty,1}^3$ and $T_{max}$ as in Proposition \ref{prop_ex_w}. Let $0<T<T_{max}$ and assume that $\sup_{t \in [0,T]} ||w(t)||_{L^\infty} \leq \frac{1}{2}$. Then, Corollary \ref{cor_ex_polar} provides
\begin{align*}
    A(t) &= S_A(t)(A_0,\phi_0)^T + \int_0^t S_A(t-s) \mathcal{N}_p(A(s),\phi(s))\,ds, \\ 
    \phi(t) &= (S_r(t)+S_p(t))(A_0,\phi_0)^T + \int_0^t (S_r(t-s)+S_p(t-s)) \mathcal{N}_p(A(s),\phi(s))\,ds.
\end{align*}
We split $\phi(t) = \sigma(t) + r(t)$  by choosing
\begin{align}
\begin{split}
    r(t) &= S_r(t)(A_0,\phi_0)^T + \int_0^t S_r(t-s) \mathcal{N}_p(A(s),\phi(s))\,ds, \\
    \sigma(t) &= S_p(t)(A_0,\phi_0)^T + \int_0^t S_p(t-s) \mathcal{N}_p(A(s),\phi(s))\,ds.
\end{split}
\label{duhamel_r_sigma}
\end{align}

\begin{proposition}
Let  $w_0 \in M_{\infty,1}^3$ and $T_{max}$ as in Proposition \ref{prop_ex_w}. Let $0<T<T_{max}$ and assume that $\sup_{t \in [0,T]} ||w(t)||_{L^\infty} \leq \frac{1}{2}$. Let $r(t)$ and $\sigma(t)$ be defined as in (\ref{duhamel_r_sigma}). Then $r,\sigma \in C([0,T_{max}), M_{\infty,1}^2) \cap C^1([0,T_{max}), M_{\infty,1})$. Furthermore, $\sigma$ is a classical solution of the equation
\begin{align*}
   \partial_t \sigma = \frac{b}{a}\sigma_{xx}+\widetilde{\mathcal{N}}_p(A,\phi,r,\sigma) - \sigma_x^2, \quad \sigma(0) = \begin{pmatrix}
        -\frac{b}{a} & 1
\end{pmatrix}(A_0,\phi_0)^T
\end{align*}
with 
\begin{align*}
   \begin{pmatrix}
       - \frac{b}{a} & 1
\end{pmatrix}\mathcal{N}_p(A,\phi) =\widetilde{\mathcal{N}}_p(A,\phi,r,\sigma)
       -\sigma_x^2, \quad \widetilde{\mathcal{N}}_p(A,\phi,r,\sigma) = 2\frac{b}{a} A_x\phi_x + A_x^2 -2\sigma_xr_x - r_x^2
\end{align*}
    
\end{proposition}
\begin{proof}
 Differentiating (\ref{duhamel_r_sigma}) two times in space and once in time, the result follows from Corollary \ref{cor_ex_polar}  by inserting $A$ and $\phi$.
\end{proof}
Now we introduce the Cole-Hopf variable
\begin{align}
\label{formula__cole_hopf_transform}
   \sigma^{\textrm{ch}}(t) = e^{-\frac{a}{b}\sigma(t)} -1
\end{align}
and record its properties.
\begin{proposition}
Let  $w_0 \in M_{\infty,1}^3$ and $T_{max}$ as in Proposition \ref{prop_ex_w}. Let $0<T<T_{max}$ and assume that $\sup_{t \in [0,T]} ||w(t)||_{L^\infty} \leq \frac{1}{2}$. Let  $\sigma^{\rm ch}(t)$ be defined as in (\ref{formula__cole_hopf_transform}). Then $\sigma^{\rm ch} \in C([0,T_{max}), M_{\infty,1}^2) \cap C^1([0,T_{max}), M_{\infty,1})$. Furthermore, $\sigma^{\rm ch}$ is a classical solution of the equation
\begin{align}
\label{equ_after_CH}
   \partial_t \sigma^{\rm ch} = \frac{b}{a}\sigma_{xx}^{\rm ch}+\widetilde{\mathcal{N}}_p(A,\phi,r,\sigma)(\sigma^{\rm ch}+1), \quad \sigma^{ch}(0) = e^{ \frac{a}{b}\begin{pmatrix}
        \frac{b}{a} & -1
\end{pmatrix}(A_0,\phi_0)^T}-1.
\end{align}
  
\end{proposition}

It is straightforward to establish the following bounds.

\begin{lemma}
Fix $K>0$. There exists some constant $C>0$ such that
\begin{align} 
\label{bound_ch_to_sigma}
||\sigma||_{M_{\infty,1}}\leq C ||\sigma^{\textrm{ch}}||_{M_{\infty,1}} \quad ||\partial_x\sigma||_{M_{\infty,1}} \leq C ||\partial_x\sigma^{\textrm{ch}}||_{M_{\infty,1}}.
\end{align}
\begin{align}
\label{bound_n_p_1}
||\widetilde{\mathcal{N}}_p(A,\phi, r, \sigma)\sigma^{\textrm{ch}}||_{M_{\infty,1}} \leq C||A||_{M_{\infty,1}^1}(||A||_{M_{\infty,1}^1} + ||\phi_x||_{M_{\infty,1}})  +   ||r_x||_{M_{\infty,1}}( ||r_x||_{M_{\infty,1}}+ ||\sigma_x^{\textrm{ch}}||_{M_{\infty,1}}),
\end{align} 
provided $||\sigma^{\textrm{ch}}||_{M_{\infty,1}}\leq \frac{1}{2}$, and
\begin{align}
\label{bound_sigma_t_ch}
||\sigma^{\textrm{ch}}||_{M_{\infty,1}} \leq C||\sigma||_{M_{\infty,1}} ,
\end{align}
provided $||\sigma||_{M_{\infty,1}}\leq K$, and
 \begin{align}
  \begin{split}
||\mathcal{N}_p(A,\phi)||_{M_{\infty,1}} &\leq C||A||_{M_{\infty,1}^1}(||A||_{M_{\infty,1}^1} + ||\phi_x||_{M_{\infty,1}})  +   ||\phi_x||_{M_{\infty,1}}^2,
\label{bound_n_p}
\end{split}\\
 \begin{split}
||\partial_x\mathcal{N}_p(A,\phi)||_{M_{\infty,1}} \leq C\Bigl (||A||_{M_{\infty,1}} &+ ||A_{xx}||_{M_{\infty,1}})||A_x||_{M_{\infty,1}} + ||A_{xx}||_{M_{\infty,1}}||\phi_x||_{M_{\infty,1}} \\
&+ (||A_x||_{M_{\infty,1}} + ||\phi_{x}||_{M_{\infty,1}})||\phi_{xx}||_{M_{\infty,1}}  \Bigr).
\label{bound_n_p_der}
\end{split}
\end{align} 
\end{lemma}
\subsection{Nonlinear iteration}
Let  $w_0 \in M_{\infty,1}^3$ and $T_{max}$ as in Proposition \ref{prop_ex_w}, and   $0<T<T_{{max}}$ with $\sup_{t \in [0,T]} ||w(t)||_{L^\infty} \leq \frac{1}{2}$. We introduce
\begin{align*}
    \eta(t) = \sup_{0\leq s\leq t}\Bigl(& (1+s)\log(2+s)^{-1}||A(s)||_{M_{\infty,1}}+(1+s)^{\frac{3}{4}} ||r_x(s)||_{M_{\infty,1}}\\
    & + (1+s)^\frac{1}{2}||\sigma_x^{\textrm{ch}}(s)||_{M_{\infty,1}} +||\sigma^{\textrm{ch}}(s)||_{M_{\infty,1}} + ||\phi(s)||_{M_{\infty,1}} +||\partial_x^3w(s)||_{M_{\infty,1}}   \Bigr)
\end{align*}
which is a continuous and monotonically increasing function in $t\in [0,T]$. Fix $t \in (0,T]$.
\bigskip

\noindent\textbf{Interpolation bounds.}
Let $0\leq s\leq t$ with $\eta(t) \leq \frac{1}{2}$.
Due to Lemma \ref{lemma_relate_pert}, we bound
\begin{align}
\label{bound_on_w_int}
||w(s)||_{M_{\infty,1}} \lesssim \sup_{0\leq s\leq t} ||(A(s),\phi(s))||_{M_{\infty,1}} \lesssim \eta(t)
\end{align}
and 
\begin{align*}  ||\partial_x^3(A(s),\phi(s))||_{M_{\infty,1}} \lesssim \sup_{0\leq s\leq t}||w(s)||_{M_{\infty,1}^3} \lesssim \eta(t).
\end{align*}
This gives
\begin{align}
    ||\partial_x A(s)||_{M_{\infty,1}}\lesssim (1+s)^{-\frac{7}{12}}\eta(t), \quad ||\partial_x^2 A(s)||_{M_{\infty,1}}\lesssim (1+s)^{-\frac{1}{4}}\eta(t)
    \label{bounds_A_x_xx}
\end{align}
using Lemma \ref{lemma_inequalites_on_mod} (ii). Furthermore,
\begin{align}
    ||\partial_x\phi(s)||_{M_{\infty,1}} \lesssim ||\partial_xr(s)||_{M_{\infty,1}} + ||\partial_x\sigma^{\rm ch}(s)||_{M_{\infty,1}} \lesssim (1+s)^{-\frac{1}{2}} \eta(t),
    \label{bound_phi_x}
\end{align}
since $\phi = r+\sigma$ and (\ref{bound_ch_to_sigma}).
Invoking Lemma \ref{lemma_inequalites_on_mod} (ii)  again yields
 \begin{align}
     ||\partial_x^2 \phi(s)||_{M_{\infty,1}}\lesssim (1+s)^{-\frac{1}{4}}\eta(t).
     \label{bound_phi_xx}
 \end{align}

\noindent\textbf{Bound on $A$.}
Let $0\leq s\leq t$ with $\eta(t) \leq \frac{1}{2}$, we have
\begin{align*}
    ||A(s)||_{M_{\infty,1}} &\lesssim (1+s)^{-1}||w_0||_{M_{\infty,1}} + \int_0^s e^{-\Theta(s-\tau)} ||\mathcal{N}_p(\tau)||_{M_{\infty,1}} \,d\tau \\
    &\qquad \qquad + \int_0^s (1+s-\tau)^{-1} ||\mathcal{N}_p(\tau)||_{M_{\infty,1}} \,d\tau \\
    &\lesssim (1+s)^{-1}||w_0||_{M_{\infty,1}} + \eta(t)^2 \int_0^{s} (1+t-s)^{-1} (1+s)^{-1} \,ds
    \\&\lesssim (1+s)^{-1} \log(2+s)(||w_0||_{M_{\infty,1}} +\eta(t)^2)
\end{align*}
where we used Propositions \ref{prop_low_fre_estimates}, (\ref{prop_se_damped_stab}), Lemma \ref{lemma_relate_pert},  (\ref{bound_n_p}), (\ref{bounds_A_x_xx}) and (\ref{bound_phi_x}).

\bigskip
\noindent \textbf{Bounds on $\sigma_x^{\rm ch}$, $r_x$, $\phi$, $\phi_x$ and $\sigma^{ch}$.}
Let $0\leq s\leq t$ with $\eta(t) \leq \frac{1}{2}$. We estimate
\begin{align*}
||\partial_x^k\sigma^{\textrm{ch}}(s)||_{M_{\infty,1}} & \lesssim (1+s)^{-\frac{k}{2}}||w_0||_{M_{\infty,1}^1} + \int_0^s (s-\tau)^{-\frac{k}{2}}||\widetilde{\mathcal{N}}_p(\tau)||_{M_{\infty,1}}||\sigma^{\rm ch}(\tau)||_{M_{\infty,1}} \,d\tau \\ 
&\lesssim (1+s)^{-\frac{k}{2}}||w_0||_{M_{\infty,1}^1} + \eta(t)^2\int_0^s (s-\tau)^{-\frac{k}{2}}(1+\tau)^{-\frac{13}{12}} \,d\tau\\
    &\lesssim (1+s)^{-\frac{k}{2}}(||w_0||_{M_{\infty,1}^1} +\eta(t)^2),
\end{align*}
using Proposition \ref{prop_low_fre_estimates} Lemma \ref{lemma_relate_pert}, (\ref{equ_after_CH}), (\ref{bound_n_p_1}), (\ref{bound_sigma_t_ch}) (\ref{bounds_A_x_xx}) and (\ref{bound_phi_x}),  and
\begin{align*}
    ||\partial_x^k r(s)||_{M_{\infty,1}} & \lesssim (1+s)^{-1}||w_0||_{M_{\infty,1}^1} + \int_0^s (s-\tau)^{-\frac{k}{2}}(1+s-\tau)^{-1}||\mathcal{N}_p(\tau)||_{M_{\infty,1}} \,d\tau \\
    &\qquad +\int_0^s e^{-\Theta(s-\tau)} ||\partial_x^ k\mathcal{N}_p(\tau)||_{M_{\infty,1}} \,d\tau\\ 
    &\lesssim (1+s)^{-1}||w_0||_{M_{\infty,1}^1} + \eta(t)^2\int_0^s (e^{-\Theta(s-\tau)}(1+\tau)^{-\frac{3}{4}}  + (s-\tau)^{-\frac{k}{2}}(1+s-\tau)^{-1}) (1+\tau)^{-1}\,d\tau
    \\
    &\lesssim (1+s)^{-\frac{3}{4}}(||w_0||_{M_{\infty,1}^1} +\eta(t)^2),
\end{align*}
using Propositions \ref{prop_low_fre_estimates}, (\ref{prop_se_damped_stab}), Lemma \ref{lemma_relate_pert},  Corollary \ref{cor_interchange_der}, (\ref{bound_n_p}), (\ref{bound_n_p_der}) (\ref{bounds_A_x_xx}), (\ref{bound_phi_x}) and (\ref{bound_phi_xx}), for $k = 0,1$. 
Finally, the bound on $||\phi||_{M_{\infty,1}}$ follows again using (\ref{bound_sigma_t_ch}) and $\phi = r+ \sigma$.

\bigskip
\noindent\textbf{Bound on $\partial_x^3w$.}
 Let $0\leq s\leq t$ with $\eta(t) \leq \frac{1}{2}$. 
 We  establish
\begin{align*}
    ||\partial_x^3w(s)|| _{M_{\infty,1}} &\lesssim ||w_0||_{M_{\infty,1}^3} + \int_0^s e^{-\Theta(s-\tau)} \eta(t)^2 \,d\tau + \int_0^s ||\partial_x^3 S_c(s-\tau)||_{M_{\infty,1} \rightarrow M_{\infty,1}} \eta(t)^2 \,d\tau \\&\lesssim ||w_0||_{M_{\infty,1}^3} + \eta(t)^2
\end{align*}
where we made use of Proposition \ref{prop_low_fre_estimates},  Corollary \ref{cor_interchange_der} and (\ref{bound_on_w_int}).

\bigskip
\noindent\textbf{Conclusion.} We have shown that there exists some constant $C\geq 1$ such that $\eta(t) \leq C(E_0 + \eta(t)^2)$ for $\eta(t)\leq \frac{1}{2}$, $t\in [0,T]$. Choosing $E_0 \leq \frac{1}{4C^2}$, we claim that $\eta(t) \leq 2CE_0$ for all $t \in [0,T_{max})$. In fact, assume that $\eta(t_0) = 2CE_0 \leq \frac{1}{2}$ for some $t_0 \in (0,T_{max})$, then 
\begin{align*}
    \eta(t_0) \leq CE_0 + C(2 CE_0)^2 < 2CE_0,
\end{align*}
which is a contradiction and therefore the result follows with $T_{max} = \infty$ thanks to Proposition \ref{prop_ex_w}.

\bigskip
\noindent\textbf{Approximation of the phase.} 
Let $t\geq 0$. From Theorem \ref{theo_stab}, we recall
\begin{align*}
    h(t) = -\frac{b}{a}\log (e^{\frac{b}{a}\partial_x^2 t}e^{\begin{pmatrix}
           1 & -\frac{a}{b}\end{pmatrix} (A_0,\phi_0)^T}).
\end{align*}
Setting $h^{\rm ch}(t) = e^{-\frac{a}{b} h(t)}-1$, we obtain 
\begin{align*}
    \partial_t h^{\rm ch} = \frac{b}{a} \partial_x^2 h^{\rm ch} , \quad h^{\rm ch}(0) = e^{\begin{pmatrix}
        1 & -\frac{a}{b}
    \end{pmatrix} (A_0,\phi_0)^T}-1.
\end{align*}
For $l= 0,1$, we first estimate
\begin{align*}
    ||\partial_x^l(\phi(t) - h(t))||_{M_{\infty,1}} \lesssim ||\partial_x^lr(t)||_{M_{\infty,1}} + ||\partial_x^l(\sigma(t) - h(t))||_{M_{\infty,1}}
\end{align*}
with 
\begin{align*}
    ||\partial_x^lr(t)||_{M_{\infty,1}} \lesssim (1+t)^{-\frac{3}{4}l} ||\partial_x^lw_0||_{M_{\infty,1}}.
\end{align*}
Furthermore,
\begin{align*}
  ||\partial_x^l(\sigma(t) - h(t))||_{M_{\infty,1}} &= \frac{b}{a}  ||\partial_x^l(\log(\sigma^{\rm ch}(t)+1) - \log(h^{\rm ch}(t)+1))||_{M_{\infty,1}} \\
  &\lesssim \begin{cases}
    ||\sigma^{\rm ch }(t) - h^{\rm ch}(t)||_{M_{\infty,1}} ,& l = 0 \\
    ||\partial_x \sigma^{\rm ch}(t)||_{M_{\infty,1}} + ||\partial_x h^{\rm ch}(t)||_{M_{\infty,1}}, & l = 1
  \end{cases}
\end{align*}
where
\begin{align*}
    ||\sigma^{\rm ch }(t) - h^{\rm ch}(t)||_{M_{\infty,1}} \lesssim \int_0^t ||\widetilde{\mathcal{N}}_p(\tau)||_{M_{\infty,1}}||\sigma^{\rm ch}(\tau)||_{M_{\infty,1}} \,d\tau \lesssim E_0^2 
\end{align*}
and
\begin{align*}
    ||\partial_x \sigma^{\rm ch}(t)||_{M_{\infty,1}} + ||\partial_x h^{\rm ch}(t)||_{M_{\infty,1}} \lesssim (1+t)^{-\frac{1}{2}} E_0 .
\end{align*}

\section{Proof of the dispersive blow-up}
\label{section_proof_blow}
\subsection{From real solution back to complex solution}
We first show that solving the integral equation 
\begin{align}
\label{original_duhamel}
     w(t) = e^{\mathcal{L}t} w_0 + \int_0^t e^{\mathcal{L}(t-s)} \mathcal{N}(w(s)) \,ds
\end{align}
leads to a solution of (\ref{duhamel_theor_blow}).

\begin{proposition}
\label{prop_from_real_to_complex}
Let $k\in\mathbb{N}_0$. Assume $w \in C([0,T),H^{k})\cap L^\infty([0,T], C_b^{k})$  is a solution of (\ref{original_duhamel}).
Then $u(x,t) = (w_1(x+2k_0t)+1+iw_2(x+2k_0t))u_p(x,t)$, $(x,t) \in \R\times[0,T]$, is  a solution of (\ref{duhamel_theor_blow}) with $u \in u_p +C([0,T),H^{k})\cap L^\infty([0,T], C_b^{k})$. In particular, if $||\partial_x^kw(t)|| \uparrow \infty$ as $t\uparrow T_0$ for some $T_0>0$, then $||\partial_x^ku(t)||_{L^\infty} \uparrow \infty$ as $t\uparrow T_0$.  
\end{proposition}
\begin{proof}
We assume first $u_p = r_0$. Let $w \in C([0,T),H^{k})\cap L^\infty([0,T], C_b^{k})$ be a solution of (\ref{original_duhamel}). Then, $w$  also solves  
\begin{align*}
   w(t) =  e^{\mathcal{D} t} w_0 + \int_0^t e^{\mathcal{D} (t-s)} (Bw(s) + \mathcal{N}(s)) \,ds
\end{align*}
with $B  = \begin{pmatrix}
    2a & 0 \\
    2b & 0
\end{pmatrix}$ and $\mathcal{D} = \begin{pmatrix}
    0 & -\partial_x^2 \\
    \partial_x^2 & 0
\end{pmatrix}$, $\mathcal{L} = \mathcal{D}+ B$, since $\mathcal{N}(s) \in L^1([0,T];H^s)$. In fact, using formula \cite[(IE)]{nagel}, we may write
\begin{align*}
    w(t) =  e^{\mathcal{D} t} w_0 &+ \int_0^t e^{\mathcal{D} (t-s)}  \mathcal{N}(s) \,ds\\
    &+\int_0^t e^{\mathcal{D}(t-s)} B e^{\mathcal{L}s}w_0 \,ds + \int_0^t\int_0^{t-s}e^{\mathcal{D}(t-s-r)} B e^{\mathcal{L}r}\mathcal{N}(s) \,dr \,ds.
\end{align*}
Using Fubini's theorem and the substitution $q = s+r$, we find
\begin{align*}
  \int_0^t\int_0^{t-s}e^{\mathcal{D}(t-s-r)} B e^{\mathcal{L}r}\mathcal{N}(s) \,dr \,ds =  \int_0^t e^{\mathcal{D}(t-q)}B\int_0^{q}  e^{\mathcal{L}r}\mathcal{N}(q-r) \,dr \,dq 
\end{align*}
and since
\begin{align*}
    e^{\mathcal{L}q}w_0 +\int_0^{q}  e^{\mathcal{L}r}\mathcal{N}(q-r) \,dr = w(q),
\end{align*}
the claim follows.
Setting $w_c = \Re(w) + i \Im(w)$, we obtain
\begin{align*}
   w_c(t) = e^{i\partial_x^2 t} w_{c,0} + \int_0^t e^{i\partial_x^2(t-s)} (\gamma w_c+\beta_0r_0^2 (2w_c+\bar{w_c}) + \beta_0r_0^2(w_c^2+2|w_c|^2 + |w_c|^2w_c))
\end{align*}
by reversing the calculations of Section \ref{section_pert_spec}.
Inserting $w_c = \frac{u}{r_0}-1$, then yields
\begin{align*}
    u(t) = e^{i\partial_x^2 t}u_0 + \int_0^t e^{i\partial_x^2(t-s)} (\gamma u(s)+ \beta|u(s)|^2u(s)) \,ds.
\end{align*}
For $u_p$ nonconstant, the last step needs slighlty more care. One  defines $\tilde{w_c}(t) = e^{-it(\partial_x + ik_0)^2 } w_c(\cdot+2k_0t,t)$  and $\tilde{u}(t) = e^{-i\partial_x^2 t} u(t)$. Then it holds $\tilde{u} = u_p(\tilde{w_c}+1)$. In particular, $\tilde{w_c} \in W^{1,\infty}((0,T); H^k \cap C_b^k )$ and one can reverse the calculations by taking the derivative in $\tilde{w}_c$, deriving the equation for $\tilde{u}$ and finally integrate in $t \in [0,T]$.
\end{proof}

\subsection{Linear estimates}
We recall from Section \ref{sec_splitting_es_on_mod} the pointwise decomposition 
\begin{align*}
    e^{\mathcal{L}t} = S_{\textrm{D}}(t) + S_{\textrm{NB}}(t)
\end{align*}
whereas $S_{\textrm{D}}(t)$ is set in  (\ref{def_s_d}) and 
  $S_{\textrm{NB}}(t) =   S_c(t) + \tilde{S}_{\textrm{NB}}(t)$
with 
\begin{align*}
    \tilde{S}_{\textrm{NB}}(t)w = \sum_{j = 2}^3 S_{e,j}(t)w  + \frac{1}{2\pi}\lim_{R \rightarrow \infty}\int_{-R}^R e^{i\xi x} r(\xi,t) \hat{w}(\xi) \,d\xi ,
\end{align*}
for $t \in [0,\infty)$. Note that $S_c(t)$ is given by (\ref{s_c_fix}), $S_{e,2,3}(t)$ by (\ref{s_e_fix}) and $r$ by (\ref{r_fix}).
  Our first purpose it to find data for which $S_{\textrm{D}}(t)$ admits blow-up.  We mainly follow the lines from \cite{bps} by choosing the suggested blow-up data. We give certain more details including the logarithmic  blow-up rate $-\log(\frac{1}{4}-t)$ as $t\uparrow \frac{1}{4}$ and the guaranteed time of existence up to $t<\frac{1}{4}$.  We employ the following basic lemma.

  \begin{lemma}
  \label{lemma_help_blow}
   There exists some constant $C>0$ and $\varepsilon>0$ such that for all  $0<a<\varepsilon$, it holds
    \begin{align*}
        |\int_0^\infty \sin(ya) \frac{1}{(1+|y|^2)^\frac{1}{2}}\,dy| \leq C(1+a\log(a))
    \end{align*}
    and 
    \begin{align*}
        |\int_0^\infty \cos(ya) \frac{1}{(1+|y|^2)^\frac{1}{2}}\,dy| \geq \frac{|\log(a)|}{C}.
    \end{align*}
  \end{lemma}
\begin{proof}
    We set 
\begin{align*}
    I(a) = \int_0^\infty \sin(ya) \frac{1}{(1+|y|^2)^\frac{1}{2}}\,dy, \quad a>0.
\end{align*}
Using the substitution $y \mapsto ay$,we rewrite
\begin{align*}
    I(a) = \int_0^\infty \frac{\sin(y)}{y}\,dy + \int_0^\infty  \sin(y)\left(\frac{1}{y} - \frac{1}{(a^2+y^2)^\frac{1}{2}}\right) \,dy = I_1(a) + I_2(a).
\end{align*}
Clearly, $|I_1(a)|\lesssim 1$ and
\begin{align*}
    |I_2(a)| \lesssim a \int_0^\infty \frac{|\sin(y)|}{y(a^2+y^2)^{\frac{1}{2}}} \lesssim a\int_0^\eta \frac{1}{(a^2+y^2)^{\frac{1}{2}}} + a\int_\eta^\infty \frac{1}{(1+\eta^2)} \lesssim a(|\log(a)|+1)
\end{align*}
for some $\eta>0$. Now, we set 
\begin{align*}
II(a) =  \int_0^\infty \cos(y) \frac{1}{(a^2+|y|^2)^\frac{1}{2}}\,dy   = \int_0^1 &\cos(y) \frac{1}{(a^2+|y|^2)^\frac{1}{2}}\,dy  + \int_1^\infty \cos(y) \frac{1}{(a^2+|y|^2)^\frac{1}{2}}\,dy \\
&= II_1(a) + II_2(a).
\end{align*}
We estimate
\begin{align*}
    |II_2(a)| \lesssim  |\int_1^\infty \sin(y) \frac{2y}{(a^2+|y|^2)^\frac{3}{2}} \,dy| + |\sin(1)|\lesssim 1
\end{align*}
and 
\begin{align*}
   | II_1(a)| =  |\int_0^1 \frac{1}{(a^2 + |y|^2)^\frac{1}{2}} \,dy| - C|\int_0^1 \frac{|y|}{(a^2 + |y|^2)^\frac{1}{2}} \,dy|
\end{align*}
for some $C>0$. Furthermore, we find
\begin{align*}
    |\int_0^1 \frac{1}{(a^2 + |y|^2)^\frac{1}{2}}| \geq |\log(a)| -C_1,
\end{align*}
for some $C_1>0$, and the estimate for $II(a)$ follows.
\end{proof}

\begin{proposition}
\label{prop_blow_very_first}
    Let $f \in C_b^1\cap L^2$  such that $ f'\in L^1$. There exists some $C>0$ such that for $g(x) = e^{-i|x|^2}f( x)$ it holds
     \begin{align}
\label{existence_in_lemma_1_4}
    e^{i\partial_x^2 t} g \in L^\infty\left(\left[0,T_0\right], C_b\right),\quad 0<T_0 < \frac{1}{4}.
    \end{align}
    Furthermore, if there exists some $K \in \R \setminus\{0\}$ such that 
    \begin{align}
    \label{relate_blow_up_data}
        f - \frac{K}{(1+|\cdot|^2)^\frac{1}{2}} \in L^1,
    \end{align}
    then there exists some $C>0$ such that
    \begin{align}
    \label{log_rate_in_lemma}
   ||e^{i\partial_x^2 t} g||_{L^\infty} \geq - C\log\left(\frac{1}{4}-t\right), \quad \text{as} \quad   t\uparrow \frac{1}{4}.
   \end{align}
\end{proposition}
\begin{proof}
We split the proof into two parts. 

\bigskip

  \noindent  \textbf{Existence for $t<\frac{1}{4}$.}
    Let $0<T_0<\frac{1}{4}$. We set $\sigma = -1$ and write 
    \begin{align*}
        e^{i\partial_x^2 t} g (x)= \frac{1}{(4i\pi t)^\frac{1}{2}} e^{\frac{i|x|^2}{4t}} \int_\R e^{iy(-\frac{2x}{4t} + y(\sigma+\frac{1}{4t}))} f(  y)\,dy
    \end{align*}
and split
  \begin{align*}
      \int_\R e^{iy(-\frac{2x}{4t} + y(\sigma+\frac{1}{4t}))} f(y) \,dy = &\int_{|y-\frac{x}{4t(\sigma + \frac{1}{4t})}|\leq t^\frac{1}{2}} e^{iy(-\frac{2x}{4t} + y(\sigma+\frac{1}{4t}))} f(  y) \,dy  \\ &+ \int_{|y-\frac{x}{4t(\sigma + \frac{1}{4t})}|> t^\frac{1}{2}} e^{iy(-\frac{2x}{4t} + y(\sigma+\frac{1}{4t}))} f(  y)\,dy =: I(f,t)(x)+ II(f,t)(x).
  \end{align*}
  We immediately obtain $||I(f,t)||_{L^\infty} \lesssim t^\frac{1}{2}||f||_{L^\infty}$. To estimate $II$, we integrate by parts yielding
  \begin{align*}
      II(f,t)(x) &= \frac{1}{i}\int_{|y-\frac{x}{4t(\sigma + \frac{1}{4t})}|> t^\frac{1}{2}}  \partial_y\left(e^{iy(-\frac{2x}{4t} + y(\sigma+\frac{1}{4t}))}\right) \frac{f(  y)}{2  (-\frac{x}{4t} + y(\sigma+\frac{1}{4t}))} \,dy \\
      &= II_1(f,t)(x) -  \frac{1}{i}\int_{|y-\frac{x}{4t(\sigma + \frac{1}{4t})}|> t^\frac{1}{2}}  e^{iy(-\frac{2x}{4t} + y(\sigma+\frac{1}{4t}))} \partial_y \left(\frac{f(  y)}{2(-\frac{x}{4t} + y(\sigma+\frac{1}{4t}))}\right) \,dy 
  \end{align*}
where 
\begin{align*}
    II_1(f,t)(x) =  \frac{-\tilde{f}_1(x,t)f( (t^{\frac{1}{2}} + \frac{x}{4t(\sigma +\frac{1}{4t})})) - \tilde{f}_2(x,t)f( (-t^{\frac{1}{2}} + \frac{x}{4t(\sigma +\frac{1}{4t})}))}{it^{\frac{1}{2}}(\sigma +\frac{1}{4t})}
\end{align*}
for some continuous and bounded function $\tilde{f}_{1,2}$ in $(x,t) \in \R\times (0,\infty)$. We compute
\begin{align*}
    &\int_{|y-\frac{x}{4t(\sigma + \frac{1}{4t})}|> t^\frac{1}{2}}  e^{iy(-\frac{2x}{4t} + y(\sigma+\frac{1}{4t}))} \partial_y \left(\frac{f(  y)}{2(-\frac{x}{4t} + y(\sigma+\frac{1}{4t}))}\right) \,dy \\
    &=  \int_{|y-\frac{x}{4t(\sigma + \frac{1}{4t})}|> t^\frac{1}{2}}  e^{iy(-\frac{2x}{4t} + y(\sigma+\frac{1}{4t}))} \frac{f'(  y)}{2(-\frac{x}{4t} + y(\sigma+\frac{1}{4t}))} \,dy \\ &\qquad - \int_{|y-\frac{x}{4t(\sigma + \frac{1}{4t})}|> t^\frac{1}{2}}   e^{iy(-\frac{2x}{4t} + y(\sigma+\frac{1}{4t}))} \frac{(\sigma + \frac{1}{4t})f(  y)}{2  (-\frac{x}{4t} + y(\sigma+\frac{1}{4t}))^{2}} \,dy\\
    &=  \int_{|y|> t^\frac{1}{2}}  e^{iy^2(\sigma + \frac{1}{4t})} \frac{f'(  y + \frac{x}{4t(\sigma + \frac{1}{4t})})}{2y(\sigma + \frac{1}{4t})} \,dy  - \int_{|y|> t^\frac{1}{2}}   e^{iy^2(\sigma + \frac{1}{4t})} \frac{(\sigma + \frac{1}{4t})f( y + \frac{x}{4t(\sigma + \frac{1}{4t})})}{2  (y(\sigma + \frac{1}{4t}))^{2}} \,dy
    \\
    &= II_2(f,t)(x) + II_3(f,t)(x).
\end{align*}

We conclude
\begin{align*}
    \frac{1}{t^{\frac{1}{2}}}\left(||II_1(f,t)||_{L^\infty} + ||II_2(f,t)||_{L^\infty} + ||II_3(f,t)||_{L^\infty}\right) \lesssim \frac{||f'||_{L^1} + ||f||_{L^\infty}}{(\sigma t + \frac{1}{4})} \lesssim ||f'||_{L^1} + ||f||_{L^\infty},
\end{align*}
for $0<t\leq T_0$.
The continuity of  $e^{i \partial_x^2 t}g $ with $0<t<\frac{1}{4}$ is a consequence of the lemma about parameter integrals. We have shown \ref{existence_in_lemma_1_4}.

\bigskip

\noindent \textbf{Determining the blow-up rate.} We consider the even initial datum $\tilde{g}(x) = e^{-i|x|^2} \frac{1}{(1+|x|^2)^\frac{1}{2}}$. Using (\ref{relate_blow_up_data}) gives 
\begin{align*}
    ||e^{i\partial_x^2 t} (g- K\tilde{g})||_{L^\infty} \lesssim t^{-\frac{1}{2}} ||g-K\tilde{g}||_{L^1}
\end{align*}
for $\frac{1}{8} \leq t < \frac{1}{4}$.
The substitution $y\mapsto y^2$ yields 
\begin{align*}
|e^{i\partial_x^2 t}\tilde{g}(0)|\gtrsim  |\int_0^\infty \cos(y(-1 + \frac{1}{4t})) \frac{1}{|y|^\frac{1}{2} }\frac{1}{(1 +|y|)^\frac{1}{2}} \,dy|- |\int_0^\infty \sin(y(-1 + \frac{1}{4t})) \frac{1}{|y|^\frac{1}{2} }\frac{1}{(1 +|y|)^\frac{1}{2}} \,dy|. 
\end{align*}

Now, we compute the difference
\begin{align}
\label{integreable_remainder_for_bessel}
    \frac{1}{|y|^\frac{1}{2} }\frac{1}{(1 +|y|)^\frac{1}{2}} - \frac{1}{(1+|y|^2)^\frac{1}{2}} =  \frac{1-|y|}{(1+|y|)^\frac{1}{2}(1+|y|^2)^\frac{1}{2} |y|^\frac{1}{2}((1+|y|^2)^\frac{1}{2} + |y|^\frac{1}{2}(1+|y|)^\frac{1}{2})}
\end{align}
and note that the function on the right hand side is integrable in $y \in (0,\infty)$.
Putting everything together, Lemma \ref{lemma_help_blow} shows (\ref{log_rate_in_lemma}).
\end{proof}
\begin{remark}
    Considering $\bar{g}$ in Proposition \ref{prop_blow_very_first}, the computations with $\sigma = 1$ show that $e^{i\partial_x^2t} \bar{g} \in L^\infty([0,\infty),C_b(\R))$.
\end{remark}
We adapt our considerations for higher derivatives.
\begin{proposition}
\label{prop_blow_higher}
  Let $\varepsilon>0$ and $k\in \mathbb{N}_0$.  Let $$u_0^\varepsilon = \varepsilon e^{-i|x|^2}\begin{cases}
        \frac{ 1 }{(1+|x|^2)^\frac{k+1}{2}}, & \textrm{ k even} \\
        \frac{x}{(1+|x|^2)^\frac{k+2}{2}}, & \textrm{ k odd}.
    \end{cases}$$  It holds $u_0^\varepsilon \in C^{k+1}_b \cap \begin{cases}
        H^{s}, & k = 0 \\
        H^{k} & k\geq 1
    \end{cases} $ for all $0<s<\frac{1}{2}$ and $\partial_x^k u_0^\varepsilon \in L^m$ for all $1<m<\infty$. We obtain 
    \begin{align*}
    e^{i\partial_x^2 t} u_0^\varepsilon \in L^\infty\left(\left[0,t_0\right], C_b^k\right), 0<T_0<\frac{1}{4}, \quad \textrm{ and } \quad    ||\partial_x^ke^{i\partial_x^2 t} u_0^\varepsilon||_{L^\infty} \gtrsim   \log(\frac{1}{4}-t)  \textrm{ as }    t\uparrow \frac{1}{4}.
    \end{align*}
\end{proposition}
\begin{proof}
Let $k \in \mathbb{N}_0$.  Using the Leibniz rule, it is easily shown that $u_0^\varepsilon \in H^{k-1}$ if $k\geq 1$ and $u_0^\varepsilon \in  H^s$ with $0<s<\frac{1}{2}$ as in \cite{bps}. There exists some constant $C_k \in \R$ such that 
\begin{align*}
   \partial_x^k u_0^\varepsilon = \varepsilon(f_1 +e^{-i|x|^2} f_2) 
\end{align*}
whereas $f_1 \in  H^1 \cap C^1_b \cap L^1$ and $$f_2(x) = \begin{cases}
    \frac{C_k}{(1+x^2)^\frac{1}{2}} \frac{x^k}{(1+x^2)\frac{k}{2}}, & \textrm{k even} \\
    \frac{C_k}{(1+x^2)^\frac{1}{2}} \frac{x^{k+1}}{(1+x^2)\frac{k+1}{2}}, & \textrm{k odd},
\end{cases}$$ for some constants $C_k \in \R \setminus \{0\}$, as simple consequence of the Leibniz rule. It holds  $ f_2 \in C^1_b\cap L^2 \cap L^m$, $1<m<\infty$, $\partial_x f_2 \in L^1$ and (\ref{relate_blow_up_data}) with $f = f_2$. Therefore, Proposition \ref{prop_blow_very_first} applies for $f_2$ and we may conclude the statement. 
\end{proof}

We summarize the crucial result for $S_\textrm{D}(t)$ in this section.
\begin{proposition}
\label{prop_blow_s_d}
    Let $\varepsilon>0$,  $k\in \mathbb{N}_0$ and $u_0^\varepsilon$ be as in Proposition \ref{prop_blow_higher}. There exists some $C>0$ such that for $$w_0^\varepsilon = \frac{1}{2}(u_0^\varepsilon+ \bar{u_0^\varepsilon},\frac{1}{i}(u_0^\varepsilon-\bar{u_0^\varepsilon}))^T,$$ it holds
    \begin{align*}
    S_{\textrm{D}}(t)w_0^\varepsilon \in L^\infty\left(\left[0,T_0\right], C_b^k\right),  0<T_0< \frac{1}{4}, \quad  \textrm{and} \quad     ||\partial_x^k S_{\textrm{D}}(t) w_0||_{L^\infty} \geq -C\log(\frac{1}{4}-t)\textrm{ as }   t\uparrow \frac{1}{4}.
    \end{align*}
\end{proposition}
\begin{proof}
    Using (\ref{S_D_rewrite}), we observe that
    \begin{align*}
        ||\partial_x^lS_{\textrm{D}}(t)w_0||_{L^\infty} = ||\partial_x^le^{at + i(b+\partial_x^2)t}u_0||_{L^\infty} = e^{at}|| e^{i\partial_x^2t} \partial_x^lu_0||_{L^\infty}, \quad l = 0,...,k,
    \end{align*}
    and hence the result follows by Proposition \ref{prop_blow_higher}.
\end{proof}
Furthermore, similarly as in (\ref{S_D_rewrite}), we rewrite
\begin{align}
    \begin{split}
         S_{\textrm{D}}(t)(f_1,f_2)^T &= \frac{
    e^{at}}{2}\left( e^{i(b+\partial_x^2)t} \begin{pmatrix}
        1 & i \\ -i & 1
    \end{pmatrix} \begin{pmatrix}
        f_1 \\f_2
    \end{pmatrix} + \overline{ e^{i(b+\partial_x^2)t} \begin{pmatrix}
        1 & i \\ -i & 1
    \end{pmatrix} \begin{pmatrix}
        f_1 \\f_2
    \end{pmatrix}}\right) \\
    &= e^{at} \Re \left( e^{i(b+\partial_x^2)t} \begin{pmatrix}
        f_1 + i f_2 \\
        -if_1 + f_2
    \end{pmatrix} \right) \\
    &= e^{at} \left( \begin{pmatrix}
        \Re(e^{i(b+\partial_x^2)t} f_1) \\ \Im(e^{i(b+\partial_x^2)t} f_1)
    \end{pmatrix} +  \begin{pmatrix}
-\Im(e^{i(b+\partial_x^2)t} f_2) \\ \Re(e^{i(b+\partial_x^2)t} f_2)
    \end{pmatrix} \right).
    \end{split}
    \label{s_d_rewrite_2}
\end{align}
for $(f_1,f_2)^T: \R\rightarrow \R^2$. This allows to adapt the dispersive estimates for $e^{i\partial_x^2 t}$.
\begin{lemma}
\label{smoothing_dispersive}
    There exists some $C>0$ such that  
    \begin{align*}
       S_{\textrm{D}}(t)f \in C_b,\quad  ||S_{\textrm{D}}(t)f||_{L^\infty} \leq Ct^{-\frac{1}{2}} ||f||_{L^1}, \quad t>0,
    \end{align*}
   for  every $f \in L^1$, 
   and
   \begin{align*}
       ||S_{\rm D} (t) f||_{L^2} \leq C  ||f||_{L^2}, \quad t\geq  0, 
   \end{align*}
   for every $f\in L^2$.
\end{lemma}

We recall the Bessel-potential spaces $H^{s,r}$, see e.g. \cite{bps}. Let $s \in (0,1)$ and $r\geq 2$. We define
\begin{align*}
 H^{s,r} =    \{f \in L^r: \mathcal{F}^{-1}((1+|\xi|^2)^\frac{s}{2} \hat{f}(\xi)) \in L^r\}
\end{align*}
and equip it with the norm
\begin{align*}
    ||f||_{H^{s,r}} = ||f||_{L^r} + ||D^sf||_{L^r}
\end{align*}
where
\begin{align*}
   D^s f =  \mathcal{F}^{-1}(|\xi|^s \hat{f}(\xi)).
\end{align*}
We are in the position to establish the estimates on the residual propagator $S_{\textrm{NB}}(t)$.

\begin{remark}
    One could also do the following analysis and arguments based on Strichartz estimates in Besov spaces for which we refer to \cite{cazenavebook}.
\end{remark}
\begin{proposition}
\label{prop_residual_blow_up}
  Let $0< s<1- \frac{1}{r}$, $1\leq m < \frac{1}{\frac{1}{r} +s}$, $2\leq r<\infty$ and  $1\leq p\leq2$.  For all $T>0$ there is some constant $C>0$ such that 
  \begin{align*}
      ||S_{\textrm{NB}}(t)||_{L^m \rightarrow H^{s,r}} \leq C
  \end{align*}
  and 
  \begin{align*}
      ||S_{\textrm{NB}}(t)||_{L^p \rightarrow C_b} \leq C
  \end{align*}
  for $t\in [0,T]$.
\end{proposition}

\begin{proof}
The estimates follow by the application of Lemmas \ref{aux_res_lemma_basic} and \ref{aux_res_lemma_h_r_s} together with the observations that we have smooth cut-offs in frequencies and the bounds (\ref{high_C_2_M_estimates}) and  (\ref{high_C_2_r_estimates})  otherwise.
\end{proof}

We prove Theorem \ref{theo_blow_up} first in the case $k\geq 1$ and then for $k=1$.

\subsection{Finite time blow-up for \texorpdfstring{$k\geq 1$}{Lg}}

We start with the local existence of a solution to (\ref{original_duhamel}) in $H^k$.

\begin{proposition} 
\label{prop_loc_k_g_1}
   Let  $k \in \mathbb{N}$. For all $T>0$ there exists a $\delta>0$ such that for all $||w_0||_{H^{k}}<\delta$ there exists a unique solution $w\in C([0,T],H^{k})$ of (\ref{original_duhamel}).
\end{proposition}
\begin{proof}
     We note that $H^{k}$ is a Banach-algebra and the proof is standard by semigroup theory, c.f. \cite{CA98}. 
\end{proof}

We establish nonlinear bounds.
\begin{lemma}
\label{bound_blow_k_1}
Fix $K>0$, $k \in \mathbb{N}$ and assume $||w||_{H^{k}} \leq K$. Then, there exists $C>0$ such that $$\sup_{l = 0,...,k}||\partial_x^l\mathcal{N}(w)||_{L^1} \leq C.$$
\end{lemma}
\begin{proof}
 The estimate immediately follows from the Cauchy-Schwarz inequality.
\end{proof}

\begin{proof}[Proof of Theorem \ref{theo_blow_up} with $k\geq 1$]
    Fix $T\geq \frac{1}{4}$ and choose $\delta$ as in Proposition (\ref{prop_loc_k_g_1}). We choose  $w_0^\varepsilon \in H^k$, $\partial_x^k w_0^\varepsilon\in L^m$ for $1<m<\infty$,   from Proposition \ref{prop_blow_s_d} such that $||w_0^\varepsilon||_{H^k} < \delta$ and $w$ is the corresponding solution of (\ref{original_duhamel}). We recall that $e^{\mathcal{L}t} = S_{\textrm{D}}(t) + S_{\textrm{NB}}(t)$. Noting that $\partial_xe^{\mathcal{L}t} = e^{\mathcal{L}t} \partial_x$ and
    using Propositions \ref{prop_blow_s_d} and \ref{prop_residual_blow_up}, we obtain
   \begin{align*}
     ||\partial_x^k e^{\mathcal{L}t}w_0^\varepsilon||_{L^\infty} \gtrsim \log(\frac{1}{4}-t)   \text{ as } t\uparrow \frac{1}{4}\quad  \text{and} \quad \sup_{0<t\leq T_0}|| e^{\mathcal{L}t}w_0^\varepsilon||_{W^{k,\infty}} \lesssim 1 \textrm{ with } 0<T_0<\frac{1}{4}
   \end{align*}
   and in particular $e^{\mathcal{L}t}w_0^\varepsilon \in C^k_b$ for $0<t<T_0$.
Using Lemmas \ref{bound_blow_k_1} and \ref{smoothing_dispersive}, and Propositions (\ref{prop_residual_blow_up}), we have
   \begin{align*}
       ||\int_0^t S_{\textrm{D}}(t-s) \partial_x^l \mathcal{N}(s) \,ds||_{L^\infty} \lesssim \int_0^t \frac{1}{(t-s)^\frac{1}{2}} ||\partial_x^l\mathcal{N}(s)||_{L^1} \,ds \lesssim \int_0^t \frac{1}{(t-s)^\frac{1}{2}} \,ds \lesssim t^\frac{1}{2}
   \end{align*}
   and
   \begin{align*}
       ||\int_0^t S_{\textrm{NB}}(t-s) \partial_x^l \mathcal{N}(s) \,ds||_{L^\infty} \lesssim \int_0^t  ||\partial_x^l\mathcal{N}(s)||_{L^1} \,ds \lesssim t
   \end{align*}
   for all $t \in [0,T]$ and $l= 0,...,k$, and in particular, 
   \begin{align*}
       \int_0^t e^{\mathcal{L}(t-s)}  \mathcal{N}(s) \,ds \in C([0,T], C_b^k).
   \end{align*}
Taking everything together, this closes the proof.
   \end{proof}

\subsection{Finite time blow-up for \texorpdfstring{$k=0$}{Lg}}

We first construct a solution to the integral system
\begin{align}
\label{integral_system}
\begin{cases}
    w_d
     = S_{\textrm{D}}(t)    w_0 + \int_0^t S_{\textrm{D}}(t-s) \mathcal{N}(w_d+w_e)(s) \,ds, \\
    w_e = S_{\textrm{NB}}(t)  w_0 + \int_0^t S_{\textrm{NB}}(t-s)  \mathcal{N}(w_d+w_e)(s) \,ds.
    \end{cases}
\end{align}

\begin{proposition}
\label{prop_k_0_loc}
   Let $\frac{3}{8}<s<\frac{1}{2}$ and $1<m<\frac{8}{7}$. For all $T>0$ there exists some $\delta>0$ such that for all $||w_0||_{H^{s}\cap L^m}<\delta$ there exists a solution $(w_d,w_e) \in C([0,T],H^{s} )\times C([0,T],L^2\cap L^\infty)$ of (\ref{integral_system}). 
\end{proposition}
\begin{proof}
We rely on Strichartz estimates for $S_\textrm{D}(t)$ and delegate its proof to Appendix \ref{appendix_proof_of_prop}.
\end{proof}

\begin{remark}
    The solution in Proposition \ref{prop_k_0_loc} is constructed and unique in the class $$C([0,T],H^{s}) \cap L^\frac{16}{3}([0,T], H^{s,8})\times C([0,T],H^{s}\cap L^\infty\cap H^{s,8}).$$
\end{remark}

We establish nonlinear bounds.
\begin{lemma}
\label{lemma_bound_k_0}
Let $ \frac{3}{8}< s<\frac{1}{2}$. Fix $K>0$ and assume $||w_d||_{H^s}+ ||w_e||_{L^\infty\cap L^2} \leq K$. Then there exists some $C>0$ such that
 \begin{align*}
     ||\mathcal{N}(w_d+w_e)||_{L^1} \leq C.
 \end{align*}
\end{lemma}
\begin{proof}
    The bounds follow by the Sobolev embeddings 
    $||u||_{L^p} \lesssim ||u||_{H^s}$
    for $p \leq \frac{1}{\frac{1}{2}-s}$.
\end{proof}

We may now proceed analogously as in the case $k\geq1$.

\begin{proof}[Proof of Theorem \ref{theo_blow_up} with $k= 0$]
 Fix $T\geq \frac{1}{4}$ and choose $\delta$ as in Proposition \ref{prop_k_0_loc}. We choose  $w_0^\varepsilon$  from Proposition \ref{prop_blow_s_d}  and recall that $w_0^\varepsilon \in H^s\cap L^m$ for $0<s<\frac{1}{2}$, $1<m < \infty$ and $w_0^\varepsilon \in C^1_b$,   such that $||w_0^\varepsilon||_{H^{\frac{15}{16}} \cap L^{\frac{9}{8}}} < \delta$. According to Proposition \ref{prop_k_0_loc}, denote $(w_d,w_e)$  the corresponding solution of (\ref{integral_system}).  Now, we choose $w = w_d+ w_e$ and observe that it satisfies
 \begin{align*}
 w(t) = e^{\mathcal{L}t} w_0^\varepsilon + \int_0^t e^{\mathcal{L}(t-s)} \mathcal{N}(w(s))  \,ds , \quad  \mathcal{N}(w(s)) =  \mathcal{N}(w_d(s)+w_e(s) ). 
 \end{align*}
 Invoking Propositions \ref{prop_blow_s_d} and \ref{prop_residual_blow_up}, it holds
   \begin{align*}
     ||e^{\mathcal{L}t}w_0^\varepsilon||_{L^\infty} \gtrsim \log(\frac{1}{4}-t)   \text{ as } t\uparrow \frac{1}{4}\quad  \text{and} \quad \sup_{0<t\leq T_0}|| e^{\mathcal{L}t}w_0^\varepsilon||_{L^\infty} \lesssim 1 \textrm{ with } 0<T_0<\frac{1}{4}
   \end{align*}
   and in particular $e^{\mathcal{L}t}w_0^\varepsilon \in C_b$ for $0<t<T_0$.
Using Lemmas \ref{lemma_bound_k_0} and \ref{smoothing_dispersive}, and Proposition \ref{prop_residual_blow_up}, we have
   \begin{align*}
       ||\int_0^t S_{\textrm{D}}(t-s) \mathcal{N}(s) \,ds||_{L^\infty} \lesssim \int_0^t \frac{1}{(t-s)^\frac{1}{2}} ||\mathcal{N}(s)||_{L^1} \,ds \lesssim \int_0^t \frac{1}{(t-s)^\frac{1}{2}} \,ds \lesssim t^\frac{1}{2}
   \end{align*}
   and
   \begin{align*}
       ||\int_0^t S_{\textrm{NB}}(t-s)  \mathcal{N}(s) \,ds||_{L^\infty} \lesssim \int_0^t  ||\mathcal{N}(s)||_{L^1} \,ds \lesssim t
   \end{align*}
   for all $t \in [0,T]$, and in particular, 
   \begin{align*}
       \int_0^t e^{\mathcal{L}(t-s)}  \mathcal{N}(s) \,ds \in C([0,T], C_b).
   \end{align*}
Taking everything together, this closes the proof.
\end{proof}

\section{Discussion and outlook}
\label{section_outlook}
\subsection{The \texorpdfstring{$L^\infty$}{Lg}-sharp stability vs. blow-up result}

The strategy of our analysis proof works equally for perturbations from $H^3$ with control in $H^3$ which  implies that we
may control $||u(t)-u_p||_{W^{2,\infty}}$ due to Sobolev embedding. In particular, this holds for sufficiently small data from  Theorem \ref{theo_blow_up}.
This begs the following question:
\begin{quote}
  \textit{  What is the minimal regularity threshold $\alpha\geq 1$ such that for all $||u_0-u_p(0)||_{C_b^{\alpha+\eta}}\ll 1$ and $\eta>0$, we have that $||u(t)||_{L^\infty}$ exists globally while $||u(t)||_{L^\infty}$ blows up in finite time for some  $||u_0-u_p(0)||_{C_b^{\alpha}}\ll 1$?}
\end{quote}
Due to our analysis in this paper, this question is positively answered with $\alpha = 1$ whenever we can reduce the regularity requirement in Theorem \ref{theo_stab} from $C^{4+\eta}$  to $C^{1+\eta}$ by the blow-up restriction from Theorem \ref{theo_blow_up}, or with $\alpha = 4$, if we find $||u_0-u_p(0)||_{C^4_b} \ll 1$ such that $||u(t)||_{L^\infty}$ blows up in finite time because of  Theorem \ref{theo_stab}.

\subsection{The \texorpdfstring{$r_0$}{Lg}-circle as possible attractor and global dynamics}
Due to the paper \cite{alex3}, the natural question arises whether we may allow for data in Theorem \ref{theo_stab} of the form $(A_0+1)e^{i \phi_0}u_p(0)$ with $||\phi_0||_{L^\infty}$ large and $||A_0||_{L^\infty}$ sufficiently small. For simplicity, we set $\omega_0 = k_0 = 0$, then the circle $C_{r_0}= \{z \in \mathbb{C}: |z|= r_0\}$ is a compact manifold which consists of equilibria. Then this question coincides with the following:
\begin{quote}
   \textit{When starting close to $C_{r_0}$, i.e., $$\textrm{dist}(C_{r_0},u_0) := \sup_{x \in \R}\inf_{c \in C_{r_0}} |c-u_0(x)| \ll 1,$$ does also the solution stay close to $C_{r_0}$?}
\end{quote}
Let $\varepsilon >0 $ and  $u_0 = (A_0+1)e^{i\phi_0}r_0$ with $||\phi_0||_{L^\infty}<\infty$. For the sake of convenience, argue
\begin{align*}
||A_0||_{L^\infty} <\varepsilon  \textrm{ if and only if } \textrm{dist}(u_0,C_{r_0}) < r_0\varepsilon  .
\end{align*}
Let $x \in \R$. We first assume $||A_0||_{L^\infty} <\varepsilon$. Then $$\inf_{\sigma \in \R}|(A_0(x)+1)e^{i\phi_0(x)}r_0-r_0 e^{i\sigma}| < r_0\varepsilon +  r_0\inf_{\sigma \in \R}|e^{i\phi_0(x)} - e^{i\sigma}| \leq r_0 \varepsilon.$$
Now, we assume that $\textrm{dist}(u_0,C_{r_0}) < r_0\varepsilon$. If $A_0 \neq 0$, then
\begin{align*}
    \inf_{\sigma \in \R }|A_0(x) + 1-e^{i (\sigma-\phi_0(x))}| < \varepsilon.
\end{align*}
It is easy to compute that the minimum is attained for $\sigma = \phi_0(x)$, and therefore, it necessarily follows $|A_0(x)|< \varepsilon$. Since this holds for all $x \in \R$, we conclude that $||A_0||_{L^\infty} < \varepsilon.$ 
It seems that new ideas need to be developed. 
Furthermore, it would be interesting whether $C_{r_0}$ can serve as as global attractor for which new ideas seem required since energy methods are not available. We aim to tackle these  questions in future.

 \subsection{Multi-d analysis for the generalized Gross-Pitaevskii equation}
 We consider the spatially higher-dimensional variant of (\ref{GP_equ})
 \begin{align*}
     \partial_t u = i\Delta u + \gamma u + \beta|u|^2u, \quad (x,t) \in \R^n\times[0,\infty),
 \end{align*}
 and the plane wave ansatz $u_p(x) = r_0 e^{ik_0\cdot x+ i \omega_0 t}$ with the analogous dispersion relation  as (\ref{relation_dispersion}). Setting again $k_0 = 0$ and $\omega_0 = 0$ for the the sake of simplicity, the ansatz in polar coordinates 
 $r(x,t)e^{i\phi(x,t)}r_0$ yields 
the equations
\begin{align*}
   A_t &= - \Delta\phi - 2\nabla A\cdot\nabla \phi + a(e^{2A}-1) \\
     \phi_t &= \Delta A + |\nabla A|^2 - |\nabla\phi|^2 + b(e^{2A}-1).
\end{align*}
The structure of these equations  gives rise to expectation that our  nonlinear stability analysis also work in higher dimensions.

\subsection{Adaptions to other dispersive-dissipative systems}
We strongly expect that  the considerations in this paper give the conceptional building block for other dispersive-dissipative equations. As an explicit example, we mention the Lugiato-Lefever equation for which we suspect analog results of Theorems \ref{theo_blow_up} and \ref{theo_stab} close to diffusively spectrally stable periodic waves. The amplitude $A$ plays the analogous role to the inverse-modulated perturbations, see e.g \cite{LLEperiod}, whose regularity is expected to be controllable via interpolation estimates using the equivalence to its forward-modulated variant. We also refer to \cite{zumbrun_remark,alex2} relying on this equivalence together with damping estimates which seem not available neither for (\ref{GP_equ}) nor for the Lugiato-Lefever equation against fully nonlocalized perturbations. Therefore, the bottleneck to adapt the results of this paper to the Lugiato-Lefever equation is the high-frequency estimation in the modulation spaces $M_{\infty,1}^m$. Here, one might think of a distributional interpretation in Bloch representation or go via the  Laplace transform, adapting the identification procedure from \cite{alex2} for the low-frequency parametrization. However, one needs a relatively accurate idea of the dispersive evolution in high-frequencies so that residuals may be estimated via an appropriate version of Lemma \ref{lemma_groe} and an exact isolation of the dispersive propagator. We expect that such  a lemma has also to be established for the general situation of Bloch- instead of Fourier diagonalization. We leave the work on these details for  future research.

\appendix

\section{Auxiliary lemmas}

\subsection{On modulation spaces}
\label{appendix_on_mod}

 Fix a nonzero Schwartz function $g$ (also called window function), and let $s\in \R$ and $p,q \in [1,\infty]$. The space $M^s_{p,q}$ consists of the tempered distributions $u$ for which 
\begin{align*}
 ||k \mapsto \langle k\rangle^s ||V_g u(\cdot,k)||_{L^p}||_{L^q} <\infty, \quad    V_g u(x,k) := \frac{1}{(2\pi)\frac{1}{2}}  \int_\R u(y) g(y-x)e^{-iky} \,dy, \quad \langle k \rangle := (1+k^2) ^\frac{1}{2}
\end{align*}
equipped with the norm
\begin{align*}
||u||_{M_{p,q}^s} :=||k \mapsto \langle k\rangle^s ||V_g u(\cdot,k)||_{L^p}||_{L^q}.   
\end{align*}

\begin{lemma}
    For all $k \in \mathbb{N}_0$ and $\eta>0$, it holds $C_b^{k+1+\eta} \hookrightarrow M_{\infty,1}^{k} \hookrightarrow C_b^{k}$. It holds $M_{\infty,1}^k$ is a Banach algebra with respect to multiplication. Furthermore, $M_{p,1} \hookrightarrow M_{p',1}$ for every $1\leq p\leq p' \leq \infty$.
\end{lemma}
\begin{proof}
 We only show $C_b^{1+\eta} \hookrightarrow M_{\infty,1}$ with $0<\eta < 1$  and refer to \cite{Leonid_thesis} for the other properties. Let $\chi$ be a cut-off function with support in $[-1,1]$ and let $f \in C_b^{1+\eta} $.  We integrate by parts
    \begin{align*}
         \int_\R f(y) g(y-x) e^{-iky} \,dy =    &(1-\chi(k)) \frac{1}{ik} \int_\R (f'(y) g(y-x)  + f(y)g'(y-x)) e^{-iky} \,dy   \\
         &+  \chi(k)\int_\R f(y) g(y-x) e^{-iky} \,dy.
         \end{align*}
         As in the proof of \cite[Lemma 2.4]{alex4}, we have
         \begin{align*}
            & \sup_{x \in \R}|(1-\chi(k)) \frac{1}{ik} \int_\R f'(y-x) g(y)  e^{-iky} \,dy|  \\
            &\lesssim (1-\chi(k)) \frac{1}{|k|} \sup_{x \in \R} \int_\R|f'(y-x)g(y) - f'(y-x+ \frac{\pi}{k})g(y + \frac{\pi}{k})| \,dy
         \end{align*}
and finally 
\begin{align*}
   & \int_\R|f'(y-x)g(y) - f'(y-x+ \frac{\pi}{k})g(y + \frac{\pi}{k})| \,dy \\ 
   &\leq  \int_\R|(f'(y-x) - f'(y-x+ \frac{\pi}{k}))g(y)| \,dy  + \int_\R|f'(y-x+ \frac{\pi}{k})(g(y + \frac{\pi}{k})-g(y))| \,dy \\
   &\lesssim \frac{1}{|k|^\eta} + \frac{1}{|k|^\frac{1}{2}},
\end{align*}
where we use
\begin{align*}
    |g(y + \frac{\pi}{k})-g(y)| \lesssim \frac{1}{|k|^\frac{1}{2}} |g(y + \frac{\pi}{k})-g(y)|^\frac{1}{2} \lesssim \frac{1}{|k|^\frac{1}{2}}( |g(y + \frac{\pi}{k})|^\frac{1}{2}+|g(y)|^\frac{1}{2})
\end{align*}
and the fact that $|g|^\frac{1}{2}$ is integrable. Since $g'$ is also a feasible window function,  the second term can be treated analogously. For the third term, the integrability is trivial.
\end{proof}

We collect properties of the Schrödinger propagator on modulation spaces.

\begin{lemma}{\cite{Leonid_thesis}}
\label{lemma_bounds}
    Let $p\in [1,\infty]$, $s\in \R$. $i\partial_x^2$ generates a strongly continuous semigroup on $M_{\infty,1}^s$ and
    \begin{align*}
        ||e^{i \partial_x^2 t}||_{M_{p,1}^s \rightarrow M_{p,1}^s} \lesssim (1+t)^{|\frac{1}{2}-\frac{1}{p}|},\quad  t\geq 0.
    \end{align*}    
\end{lemma}

We adapt  the following useful inequalities on $M_{\infty,1}$.

\begin{lemma}
\label{lemma_inequalites_on_mod}
    \begin{enumerate}
        \item (Young's inequality for convolutions.) There exists some $C>0$ such that for all $h \in L^1$ and $f \in M_{\infty,1}$, it holds 
        \begin{align*}
            ||f\ast h||_{M_{\infty,1}} \leq C ||h||_{L^1}||f||_{M_{\infty,1}}.
        \end{align*}
        \item (A Gagliardo-Nirenberg-type inequality.) Let $k,l \in \mathbb{N}_0$, $k\neq 0$, and $0 \leq l \leq k$. There exists some $C>0$ such that for all $f \in M_{\infty,1}^k$, it holds 
        \begin{align*}
            ||\partial_x^lf||_{M_{\infty,1}} \lesssim ||\partial_x^kf||_{M_{\infty,1}}^{\frac{l}{k}}||f||_{M_{\infty,1}}^{1-\frac{l}{k}} . 
        \end{align*}
        \item (Scaling upper bound.) There exists some $C>0$ such that for all $\lambda >0$, it holds
        \begin{align*}
        ||f(\lambda \cdot)||_{M_{\infty,1}} \leq  C||f||_{M_{\infty,1}} \begin{cases}
            \lambda, &  \lambda \geq 1   \\
            1, & 0<\lambda \leq 1.
        \end{cases}
    \end{align*}
    \end{enumerate}
\end{lemma}
\begin{proof} To prove (i), we estimate
\begin{align*}
    ||\int_\R \int_\R f(y-z) h(z) \,dz \,g(y-x) e^{-iky} \,dy||_{L^\infty(x)} &= ||\int_\R h(z)\int_\R f(y-z) g(y-x) e^{-iky} \,dy \,dz||_{L^\infty(x)} \\ &\lesssim ||h||_{L^1}||\int_\R f(y-z) g(y-x) e^{-iky} \,dy||_{L^\infty(x) L^\infty(z)} \\
    &\lesssim ||h||_{L^1}||\int_\R f(y) g(y-x+z) e^{-iky} \,dy||_{L^\infty(x) L^\infty(z)} \\
    &= ||h||_{L^1}||\int_\R f(y) g(y-x) e^{-iky} \,dy||_{L^\infty(x)}
\end{align*}
by the classical Young's inequality, the substitution rule and Fubini's theorem. Finally, we arrive at
\begin{align*}
  ||f\ast h||_{M_{\infty,1}} = ||\int_\R \int_\R f(y-z) h(z) \,dz \,g(y-x) e^{-iky} \,dy||_{L^\infty(x) L^1(k)} \lesssim ||h||_{L^1} ||f||_{M_{\infty,1}}.
\end{align*}
    We prove   (ii) for $k = 2$ and $l=1$. Let $f \in M_{\infty,1}^2$. We argue as in \cite[Example 2.9]{pazy}. We have by Taylor's formula, for every $h>0$,
    \begin{align*}
        f(h+\cdot) = f + h f' + \int_0^h (h-s)f''(\cdot+s) \,ds 
    \end{align*}
    and therefore, by the translation invariance of the $M_{\infty,1}$-norms that
    \begin{align*}
        ||f'||_{M_{\infty,1}} \leq 2h^{-1}||f||_{M_{\infty,1}} + 2^{-1}h||f''||_{M_{\infty,1}}.
    \end{align*}
    Setting $h = ||f||_{M_{\infty,1}}^{\frac{1}{2}} ||f''||_{M_{\infty,1}}^{-\frac{1}{2}}$ yields the result.
    The remaining cases follow analogously. For the proof (iii) and further bounds under dilation scaling, we refer to \cite{mod_scaling}.
\end{proof}
\subsection{Estimations on Fourier multiplier}
\label{appendix_mulitplier}
The next lemma is a corollary of the last lemma together with \cite[Lemma A.1]{bjorn1}.
\begin{lemma} \label{lemma_low_frequency}
Let $n \in \mathbb N$, $m \in \mathbb N_0$ and $\xi_0 > 0$. Let $\lambda \in C^4\big((-\xi_0,\xi_0),\C\big)$ and $\chi \in C_b^4\big(\R^2,\C^{n \times n}\big)$. Suppose there exist constants $C, \mu > 0$ such that it holds
\begin{itemize}
\item[i)] $\lambda'(0) \in \ri \R$;
\item[ii)] $\Re \, \lambda(\xi) \leq -\mu \xi^2$ for all $\xi \in (-\xi_0,\xi_0)$;
\item[iii)] $\overline{\mathrm{supp}(\chi(\cdot,x))} \subset (-\xi_0,\xi_0)$ for all $x \in \R$;
\item[iv)] $\displaystyle\sup_{x \in \R} \left\|\chi(\cdot,x)\right\|_{W^{4,\infty}} \leq C$.
\end{itemize}
Let $a \in \R$ be such that $\lambda'(0) = a\ri$ and $v \in M_{\infty,1}$. Then, the pointwise estimate
\begin{align*}
\left|\left|\int_\R\int_\R \re^{t \lambda(\xi)} \xi^m \chi(\xi,\cdot-y) \re^{\ri \xi(\cdot-y)} \,d \xi v(y)\,dy \right| \right|_{M_{\infty,1}}&\lesssim t^{-\frac{m}{2}} ||v||_{M_{\infty,1}}
\end{align*}
holds  for each $v \in M_{\infty,1}$ and $t \geq 1$. 
\end{lemma}
\begin{proof}
    Under the assumptions (i)-(iv), it is shown in \cite{bjorn1} that the pointwise estimate holds
\begin{align*}
\left|\int_\R \re^{t \lambda(\xi)} \xi^m \chi(\xi,x-y) \re^{\ri \xi(x-y)} \de \xi \right| &\lesssim t^{-\frac{m+1}{2}}\left(1 + \frac{(x-y + a t)^4}{t^2}\right)^{-1},
\end{align*}
 for each $t \geq 1$ and $x,y\in \R$.  Invoking Lemma \ref{lemma_inequalites_on_mod} (i), we arrive at the claimed estimate.
\end{proof}

\begin{lemma}
\label{lemma_high-frqu_para}
Let $\xi_0,d > 0$ and $a \in \R$. Let $F \in C_{\mathrm{ub}}^2(\R,\R)$ be supported on $\R \setminus [-\xi_0,\xi_0]$. Then, there exists a constant $\mu_0 > 0$ such that we have
\begin{align*}
\left\|\int_\R\int_\R \re^{\ri \xi(\cdot-y) + \left(a \ri \xi - d\xi^2\right)t} F(\xi) \de \xi\, v(y) \de y\right\|_{M_{\infty,1}} \lesssim  \left(1 + \frac{1}{\sqrt{t}}\right)\re^{-\mu_0 t} \|v\|_{M_{\infty,}}
\end{align*}
for each $v \in M_{\infty,1}$ and $t > 0$.
\end{lemma}
\begin{proof}
    In \cite{bjorn1}, it is shown that 
    \begin{align*}
        |\int_\R\re^{\ri \xi y + \left(a \ri \xi - d\xi^2\right)t} F(\xi) \de \xi| \lesssim (1+\frac{1}{\sqrt{t}})e^{-\mu_0t} \frac{1}{1+(y+at)^2}.
    \end{align*}
    Therefore, the result follows immediately from Lemma \ref{lemma_inequalites_on_mod} (i). 
\end{proof}
From now on we write
\begin{align*}
    H_\sigma v = \int_\R e^{i\xi \cdot} \sigma(\xi) \hat{v}(\xi) \,d\xi
\end{align*}
for a Fourier multiplier.
To treat high-frequency contributions, we also invoke \cite[Lemma 8]{groe}.

\begin{lemma}
\label{lemma_groe}
It holds
\begin{align*}
    ||H_\sigma v||_{M_{\infty,1} \lesssim }||\sigma||_{M_{\infty,1}} ||v||_{M_{\infty,1}}
\end{align*}
whenever $\sigma \in M_{\infty,1}$ and $v \in M_{\infty,1}$. 
\end{lemma}

\begin{proof}
    We recall its elegant  proof from \cite{groe} for the sake of   self-containment. We introduce the space $W_g(\mathcal{F}L^1,l^\infty)$ consisting of every function $\sigma \in L^\infty$ for which 
    \begin{align*}
        ||\sigma||_{W_g(\mathcal{F}L^1,l^\infty)} = \sup_{x \in \R} \int_{\R} |V_g\sigma(x,k)| \,dk < \infty.
    \end{align*}
    Clearly, $||\sigma||_{W_g(\mathcal{F}L^1,l^\infty)} \leq ||\sigma||_{M_{\infty,1}}$.
    Let $g$ be a real Schwartz function with the property $g(x) = g(-x)$, then $g_1 = g\ast g$ is also a feasible window function. Notice that the norms $||\cdot||_{M_{\infty,1}}$ are equivalent for different window functions. Rewriting
    \begin{align*}
        V_g\sigma(-x,k) = e^{-ikx}(\sigma \ast M_kg)(x), \quad (M_kg)(x) = e^{ikx}g(x).
    \end{align*}
    We further note $M_k g_1 = M_kg \ast M_k g$. We  observe $H_\sigma v = \check{\sigma}\ast v$. We find
    \begin{align*}
        ||H_\sigma v||_{M_{\infty,1}} &\lesssim \int_\R ||H_\sigma v \ast M_k g_1||_{L^\infty} \,dk \\ &= \int_\R ||\check{\sigma} \ast v \ast M_k g \ast M_k g||_{L^\infty} \,d k \\
        &\leq \int_\R ||\check{\sigma} \ast M_k g||_{L^1}  ||v  \ast M_k g||_{L^\infty} \,d k \\
         &\leq \sup_{k \in \R} ||\check{\sigma} \ast M_k g||_{L^1}  \int_\R  ||v  \ast M_k g||_{L^\infty} \,d k  = \sup_{k \in \R} ||\check{\sigma} \ast M_k g||_{L^1} ||v||_{M_{\infty,1}}. 
    \end{align*}
    Now, since
    \begin{align*}
        |\check{\sigma} \ast M_k g| = |\mathcal{F}^{-1}(\sigma \widehat{M_kg})| = |\mathcal{F}^{-1}(\sigma \hat{g}(\cdot+k))| =  |V_{\hat{g}}\sigma(-k,x)|
    \end{align*}
    and $\hat{g}$ is a feasible window function, we obtain 
    \begin{align*}
      \sup_{k\in \R}  ||\check{\sigma} \ast M_k g||_{L^1} \lesssim ||\sigma||_{W_g(\mathcal{F}L^1,l^\infty)}.
    \end{align*}
    All in all, we have shown the claim.
\end{proof}

The next two lemmas are used to control residual contributions  in the proof of the dispersive blow-up.

\begin{lemma}[Residual estimates I]
\label{aux_res_lemma_basic}
   Let  $||\sigma||_{L^1\cap L^2} \leq C$, then 
   \begin{align*}
       ||H_\sigma v||_{L^\infty} \lesssim C||v||_{L^p}
   \end{align*}
   for all $v\in L^p$ with $1\leq p\leq 2$. In particular, $H_\sigma v$ is continuous.
\end{lemma}
\begin{proof}
    In particular,  we have
    \begin{align*}
        ||\check\sigma||_{L^q} \lesssim ||\sigma||_{L^1\cap L^2} \lesssim C
    \end{align*}
    for any $q\geq 2$. Setting $\frac{1}{p} + \frac{1}{q} = 1$,
we conclude
    \begin{align*}
        ||H_\sigma v||_{L^\infty}  \lesssim ||\check{\sigma} \ast v||_{L^\infty} \lesssim ||\sigma||_{L^1\cap L^2} ||v||_{L^p}
    \end{align*}
    by Young's inequality for convolutions. The continuity follows from the observation that $\sigma \hat{v} \in L^1$ and the Riemann-Lebesgue lemma.
\end{proof}
We have defined the Bessel-potential spaces $H^{s,r}$ and $D^s$ in Section \ref{section_proof_blow} and now establish the key Lemma for the proof of Proposition \ref{prop_residual_blow_up}.
\begin{lemma}[Residual estimates II]
\label{aux_res_lemma_h_r_s}
 Let $2\leq r < \infty$,  $0<s<1-\frac{1}{r}$ and $1\leq q < \frac{1}{\frac{1}{r} +s}$. Let $\sigma: \R \rightarrow \C$  and assume that $|\partial_\xi^k \sigma(\xi)| \leq C (1+|\xi|)^{-2+k}$, $k = 0,1$. 
 it holds
 \begin{align*}
     ||H_\sigma v||_{H^{s,r}} \lesssim C||v||_{L^q}
 \end{align*}
 for $v \in L^q$.
\end{lemma}
\begin{proof}
We write
\begin{align*}
    ||D^s H_\sigma v||_{L^r} = || \mathcal{F}^{-1}(|\xi|^s \sigma (\xi))\ast v||_{L^r}.
\end{align*}
Let $\chi$ be a smooth cut-off function with support in $(-1,1)$. By Young's inequality for convolutions, we find
\begin{align*}
    ||D^s H_\sigma v||_{L^r} \lesssim  \Bigr(|| \mathbf{1}_{\{|x|\geq  1\}}&\mathcal{F}^{-1}((1-\chi(\xi))|\xi|^s \sigma(\xi))||_{L^p} \\
    &+ || \mathbf{1}_{\{|x|\leq 1\}}\mathcal{F}^{-1}((1-\chi(\xi))|\xi|^s\sigma(\xi))||_{L^p} \\
    & \qquad \quad +  || \mathbf{1}_{\{|x|\geq  1\}}\mathcal{F}^{-1}(\chi(\xi)|\xi|^s \sigma(\xi))||_{L^p}\Bigr)||w||_{L^q}
\end{align*}
with $\frac{1}{p} + \frac{1}{q} = 1+ \frac{1}{r}$. Now, we set $G(\xi) = (1-\chi(\xi))|\xi|^s \sigma(\xi)$ with 
\begin{align*}
    \sup_{\eta \in [\xi,\xi+\frac{1}{x}]} |G'(\eta)| \lesssim C (1+\xi^2)^{\frac{s-1}{2}}
\end{align*}
for $|x|\geq1 $. We estimate
\begin{align*}
    ||\int_\R e^{i\xi x} \mathbf{1}_{\{|x|\geq  1\}}G(\xi) \,d\xi||_{L^p} &\leq ||\int_\R \mathbf{1}_{\{|x|\geq  1\}} | G(\xi) - G\left(\xi+\frac{1}{x}\right)| \,d\xi ||_{L^p} \\&\lesssim C\int_\R (1+|\xi|^{2-s})^{-(1-\sigma)}(1+|\xi|^{s-1})^\sigma \,d\xi  ||\frac{\mathbf{1}_{\{|x|\geq  1\}}}{(1+|\cdot |^\sigma)}||_{L^p}.
\end{align*}
 Our condition is now
\begin{align*}
  \sigma p>1, \quad (2-s)(1-\sigma) + \sigma(1-s)>1 \Longleftrightarrow
 1+ \frac{1}{r} - \frac{1}{q} < \sigma < 1-s.
\end{align*}
Resolving for $q$ gives the sufficient condition
\begin{align*}
    1\leq q< \frac{1}{1+\frac{1}{r} - (1-s)} = \frac{1}{\frac{1}{r} +s}.
\end{align*}
Note that
\begin{align*}
    || \mathbf{1}_{\{|x|\leq 1\}}\mathcal{F}^{-1}((1-\chi(\xi))|\xi|^s\sigma(\xi))||_{L^p} \lesssim ||(1+(\cdot)^2)^{\frac{s}{2}}\sigma||_{L^1} \lesssim C
\end{align*}
and
\begin{align*}
   || \mathbf{1}_{\{|x|\geq 1\}}\mathcal{F}^{-1}(\chi(\xi)|\xi|^s\sigma(\xi))||_{L^p} \lesssim ||\partial_\xi(\chi(\xi)\xi^s\sigma(\xi))||_{L^1} ||\frac{\mathbf{1}_{\{|x|\geq 1\}}}{1+|\cdot|}||_{L^p} \lesssim C
\end{align*}
using integration by parts for the latter and  $p>1$.
Arguing analogously, one obtains that a sufficient condition to guarantee
\begin{align*}
    ||H_F w||_{L^r} \lesssim C ||w||_{L^q},
\end{align*} is
\begin{align*}
    1+ \frac{1}{r} - \frac{1}{q} < 1.
\end{align*}
All in all, this concludes the proof.
\end{proof}

\section{Proof of Proposition \ref{prop_k_0_loc}}
\label{appendix_proof_of_prop}

We perform a fixed point argument using Strichartz estimates on the propagator $S_{\textrm{D}}(t)$. 

\begin{lemma}[Strichartz estimates for $S_{\textrm{D}}(t)$]  Let $T>0$.  Let $(p,q)$ be admissible, i.e., $\frac{2}{q} = \frac{1}{2} - \frac{1}{p}$. 
There exists a $T$-dependent constant $C_T>0$ such that 
\begin{align}
\label{strichartz2}
     ||S_{\textrm{D}}(t) f||_{L^q([0,T],L^p(\R))} \leq C_T||f||_{L^2},
\end{align}
for any $f \in L^2(\R)$, and
\begin{align}
\label{strichartz3}
    ||\int_0^t S_{\textrm{D}}(t-s) g(x,s) \,ds||_{L^q_t([0,T], L^p_x(\R))} \leq C_T ||g(x,t)||_{L^{q'}_t([0,T], L^{p'}_x(\R))},
\end{align}
for $g \in L^{q'}([0,T], L^{p'})$. Here, $\frac{1}{p} + \frac{1}{p'}  = \frac{1}{q} + \frac{1}{q'} = 1$.
\end{lemma}
\begin{proof}
Recalling (\ref{s_d_rewrite_2}) reveals
\begin{align*}
    |S_{\textrm{D}}(t)(f_1,f_2)| \lesssim_{T} |e^{i\partial_x^2 t}f_1| + |e^{i\partial_x^2 t}f_2|
\end{align*}
and
\begin{align*}
 | \int_0^t S_{\textrm{D}}(t-s) (g_1(x,s), g_2(x,s)) \,ds  | \lesssim_T & |\int_0^t e^{i\partial_x^2(t-s)} (e^{-ibs -as}g_1(x,s)) \,ds | \\
 &+ |\int_0^t e^{i\partial_x^2(t-s)} (e^{-ibs-as}g_2(x,s)) \,ds |,
\end{align*}
for all $0\leq t\leq T$. Therefore, the Strichartz estimates follow by the standard ones on the Schrödinger propagator, see e.g. \cite{bps}.
\end{proof}
Fix $w_0 \in H^s \cap L^m$ and $T>0$. Now, we aim to perform a fixed point argument on the mapping $\mathcal{T}$ defined as
\begin{align*}
   (w_d,w_d) \mapsto \begin{cases}
     S_{\textrm{D}}(t)    w_0 + \int_0^t S_{\textrm{D}}(t-s) \mathcal{N}(w_d+w_e)(s) \,ds, \\
     S_{\textrm{NB}}(t)  w_0 + \int_0^t S_{\textrm{NB}}(t-s)  \mathcal{N}(w_d+w_e)(s) \,ds
    \end{cases} 
\end{align*}
in the space
\begin{align*}
    W_T^{s,r,q} = (C([0,T];H^{s}(\R)) \cap L^q([0,T], H^{s,r}(\R))\times C([0,T];H^{s}(\R)\cap L^\infty(\R)\cap H^{s,r}(\R))
\end{align*}
for suitable choices of $q,r$ and $s$. Since we only need the fixed point argument to close for a fixed choice, we set $$q= \frac{16}{3}, \quad r = 8, \quad \frac{3}{8}<s<\frac{1}{2} \quad \textrm{and} \quad 1<m< \frac{8}{7}.$$ Then $(r,q)$ is admissible. In particular, we have the estimates
\begin{align}
\label{lp_embedding}
    ||f||_{L^p} \lesssim ||f||_{H^s},
\end{align}
for any $2\leq p \leq 8  < \frac{2}{1-2s}$ by Sobolev embedding. We recall the 
Kato-Ponce inequality 
\begin{align}
\label{KPI}
    ||D^s (fg)||_{L^r} \lesssim ||D^sf||_{L^{q_1}}||g||_{L^{p_1}} + ||f||_{L^{q_2}}||D^sg||_{L^{p_2}}.
\end{align}
with $\frac{1}{r} = \frac{1}{p_1} + \frac{1}{q_1} = \frac{1}{p_2} + \frac{1}{q_2}$, see e.g. \cite{kp_ref}. We establish the following bounds and estimates. 

\begin{lemma}
\label{prop_beweis_lemma_1}
 Let  $w_0 \in H^s\cap L^m$ and $T>0$.  Fix $K>0$ and assume that $||(w_d^1,w_e^1)||_{W^{s,r,q}_T}, (w_d^2,w_e^2)||_{W^{s,r,q}_T} \leq K$. There exists some $T$-dependent constant $C_T>0$ such that
\begin{align}
   \begin{split}
  \sup_{0\leq \tau\leq T}  &||\mathcal{T}_e(w_d^1(\tau),w_e^1(\tau))||_{H^s\cap L^\infty} \\
  &\leq C_T\left( ||w_0||_{L^m}+ K\left(  \sup_{0\leq \tau\leq T}||w_e^1(\tau)||_{L^2\cap L^\infty} +   \sup_{0\leq \tau\leq T}||w_d^1(\tau)||_{L^2}\right)\right),
  \end{split}
  \label{t_bound_1}
   \end{align}

      \begin{align}
   \begin{split}
   \sup_{0\leq \tau\leq T} & ||\mathcal{T}_d(w_d^1(\tau),w_e^1(\tau))||_{L^2} \\ &\leq C_T\left(||w_0||_{L^2} + K\left( \sup_{0\leq \tau\leq T} ||w_e^1(\tau)||_{L^2\cap L^\infty} +  \sup_{0\leq \tau\leq T} ||w_d^1(\tau)||_{H^s}\right)\right).  \label{t_bound_2}  \end{split}
   \end{align}

   and
   
   \begin{align}
   \begin{split}
  \sup_{0\leq \tau\leq T}  &||\mathcal{N}(w_d^1(\tau)+w_e^1(\tau))- \mathcal{N}(w_d^2(\tau)+w_e^2(\tau))||_{L^1} \\
  &\leq C_T K\left(  \sup_{0\leq \tau\leq T}||w_e^1(\tau)-w_e^2(\tau)||_{L^2\cap L^\infty} +   \sup_{0\leq \tau\leq T}||w_d^1(\tau)-w_d^2(\tau)||_{L^2}\right),
  \label{c_bound_1}
  \end{split}
   \end{align}

      \begin{align}
   \begin{split}
   \sup_{0\leq \tau\leq T} & ||\mathcal{N}(w_d^1(\tau)+w_e^1(\tau))- \mathcal{N}(w_d^2(\tau)+w_e^2(\tau))||_{L^2} \\ &\leq C_TK\left( \sup_{0\leq \tau\leq T} ||w_e^1(\tau)-w_e^2(\tau)||_{L^2\cap L^\infty} +  \sup_{0\leq \tau\leq T} ||w_d^1(\tau)-w_d^2(\tau)||_{H^s}\right).  
   \label{c_bound_2}
   \end{split}
   \end{align}
  \end{lemma}
  \begin{proof}
      To simplify the notation without losing essential information, we write $w_d$ and $w_e$ without distinguishing their components. The nonlinearity contains terms of the form
      \begin{align*}
          w_dw_e, w_d^2, w_e^2, w_d^2w_e, w_dw_e^2, w_e^3, w_d^3.
      \end{align*}
       We only consider the terms $w_d^2$ and $w_d^3$ since for the others, we may use that $w_e \in L^2\cap L^\infty$. We bound
      \begin{align*}
          ||w_d^3||_{L^1} + ||w_d^3||_{L^1}  \lesssim ||w_d||_{L^3}^3 + ||w_d||_{L^2}^2  \lesssim ||w_d||_{H^s}^3 + ||w_d||_{H^s}^2
      \end{align*}
      and
      \begin{align*}
          ||w_d^3||_{L^2} + ||w_d^2||_{L^2}  \lesssim ||w_d||_{L^6}^6 + ||w_d||_{L^4}^4  \lesssim ||w_d||_{H^s}^6 + ||w_d||_{H^s}^4.
      \end{align*}
      Furthermore, we use Proposition \ref{prop_residual_blow_up} to estimate $||S_{\textrm{NB}}(t)w_0||_{H^s\cap L^\infty} \lesssim_t ||w_0||_{L^m}$ and  $||S_{\textrm{D}}(t)w_0||_{L^2} \lesssim_t ||w_0||_{L^2}$.
  \end{proof}
An immediate consequence of the latter lemma is the following.
\begin{lemma}
\label{upgrade_lemma_w_e}
    Let  $w_0 \in  L^m$ and $T>0$. Fix $K>0$ and assume $||(w_d^1,w_e^1)||_{W^{s,r,q}_T}, (w_d^2,w_e^2)||_{W^{s,r,q}_T} \leq K$. For all $\frac{3}{8}<\tilde{s}<\frac{1}{2}$ and $2\leq \tilde{r}\leq 8$, there exists some $T$-dependent constant $C_T>0$ such that
   \begin{align*}
       \sup_{0\leq \tau\leq T}||\mathcal{T}_e(w_d^1(\tau),w_e(\tau)^1) - \mathcal{T}_e(w_d^2(\tau),w_e^2(\tau)) ||_{H^{\tilde{s},\tilde{r}}} \leq C_T K  ||(w_e^1,w_d^1)- (w_e^2,w_d^2)||_{W_T^{s,r,q}},
   \end{align*}
   and
   \begin{align*}
       \sup_{0\leq \tau\leq T}||\mathcal{T}_e(w_d^1(\tau),w_e^1(\tau))||_{H^{\tilde{s},\tilde{r}}} \leq C_T  (||w_0||_{L^m} + K||(w_e^1,w_d^1)||_{W_T^{s,r,q}}).
   \end{align*}
\end{lemma}
\begin{proof}
This follows with the established bounds in Lemma \ref{prop_beweis_lemma_1} and with Proposition \ref{prop_residual_blow_up}.   
\end{proof}
 Next we find the following bounds with the help of the Strichartz estimates. 
\begin{lemma}
 Let  $w_0 \in H^s$ and $T>0$ and assume $||(w_d^1,w_e^1)||_{W^{s,r,q}_T}, (w_d^2,w_e^2)||_{W^{s,r,q}_T} \leq K$. There exists some $T$-dependent constant $C_T>0$ such that
   \begin{align}
   \begin{split}
       ||\int_0^t &S_D(t-\tau)(D^s(\mathcal{N}(w_d^1(\tau)+w_e^1(\tau))- \mathcal{N}(w_d^2(\tau)+w_e^2(\tau)))) \,d\tau||_{L^\infty_t([0,T], L^2(\R))} \\
       &\leq C_T K||(w_e^1,w_d^1)- (w_e^2,w_d^2)||_{W_T^{s,r,q}},
       \label{c_bound_3}
         \end{split}
   \end{align}

\begin{align}
   \begin{split}
       ||\int_0^t &S_D(t-\tau)(\mathcal{N}(w_d^1(\tau)+w_e^1(\tau))- \mathcal{N}(w_d^2(\tau)+w_e^2(\tau))) \,d\tau||_{L^q_t([0,T], L^r(\R))} \\
       &\leq C_T K||(w_e^1,w_d^1)- (w_e^2,w_d^2)||_{W_T^{s,r,q}},
       \label{c_bound_4}
         \end{split}
         \end{align}         
   and   
\begin{align}
   \begin{split}
       ||\int_0^t &S_D(t-\tau)(D^s(\mathcal{N}(w_d^1(\tau)+w_e^1(\tau))- \mathcal{N}(w_d^2(\tau)+w_e^2(\tau)))) \,d\tau||_{L^q_t([0,T], L^r(\R))} \\
       &\leq C_T K||(w_e^1,w_d^1)- (w_e^2,w_d^2)||_{W_T^{s,r,q}}.
       \label{c_bound_5}
         \end{split}
         \end{align}
         Furthermore, we have
         \begin{align}
             \sup_{0\leq \tau\leq T}||\mathcal{T}_d(w_d^1(\tau),w_e^1(\tau))||_{H^{s,r}} \leq C_T \left(||w_0||_{H^s}+K  ||(w_e^1,w_d^1)||_{W_T^{s,r,q}}\right)
             \label{t_bound_3}
         \end{align}
         and
         \begin{align}
             \sup_{0\leq \tau\leq T}||\mathcal{T}_d(w_d^1(\tau),w_e^1(\tau))||_{H^{s}} \leq C_T \left(||w_0||_{H^s}+K   ||(w_e^1,w_d^1)||_{W_T^{s,r,q}}\right).
             \label{t_bound_4}
         \end{align}
\end{lemma}
\begin{proof}
 For the first estimate, we only discuss the nonlinear terms $w_d^2$, $w_d^3$ as the others can be estimated similarly using the $L^\infty\cap L^2$-bounds on $w_e$ and the $H^{s,r}$-bound for $r = 8$ due to Lemma \ref{upgrade_lemma_w_e}. For $k= 1,2$, we have the bounds
 \begin{align*}
     &||\int_0^t S_D(t-\tau) D^s(w_d^{k+1}(\tau))\,d\tau||_{L^\infty L^2} \lesssim \int_0^T ||D^s w_d^{k+1}(t)||_{L^2} \,dt \\
     &\lesssim \int_0^T ||D^sw_d(t)||_{L^8} ||w_d(t)^k||_{L^{\frac{8}{3}}} \,dt\lesssim_T  ||D^s w_d||_{L^\frac{16}{3} L^8} ||w_d||_{L^\infty H^s}^k,
 \end{align*}
\begin{align*}
     &||\int_0^t S_D(t-\tau) D^s (w_d^{k+1}(\tau))\,d\tau||_{L^\frac{16}{3} L^8} \lesssim (\int_0^T ||D^s w_d^{k+1}(t)||_{L^{\frac{8}{7}}}^{\frac{16}{13}} \,dt)^\frac{13}{16} \\
     &\lesssim (\int_0^T ||D^sw_d(t)||_{L^2}^\frac{16}{13} ||w_d(t)^k||_{L^\frac{8}{3}}^\frac{16}{13}\,dt)^\frac{13}{16} \lesssim  ||w_d||_{L^\infty H^s}^{k+1}
 \end{align*} 
and
\begin{align*}
     &||\int_0^t S_{\textrm{D}}(t-\tau)  (w_d^{k+1}(\tau))\,d\tau||_{L^\frac{16}{3} L^8} \lesssim (\int_0^T ||w_d^{k+1}(t)||_{L^{\frac{8}{7}}}^{\frac{16}{13}} \,dt)^\frac{13}{16} \\
     &\lesssim  ||w_d||_{L^\infty H^s}^{k+1}
 \end{align*} 
 where we used (\ref{strichartz3}) and (\ref{KPI}).
 One immediately estimates the initial contribution in $H^{s,r}$ using (\ref{strichartz2}).
\end{proof}

We are now in the position to give the

\begin{proof}[Proof of Proposition \ref{prop_k_0_loc}]
Fix $T>0$. We claim that there exists some $\varepsilon>0$ and $K>0$ such that $$\mathcal{T}: \{||f||_{W_T^{s,r,q}}\leq K \} \rightarrow \{||f||_{W_T^{s,r,q}}\leq K \}$$  is well-defined and defines a contraction. Using (\ref{t_bound_1}), (\ref{t_bound_2}), (\ref{t_bound_3}), (\ref{t_bound_4}) and Lemma \ref{upgrade_lemma_w_e}, we first observe that
\begin{align*}
  ||T(w_d,w_e)||_{W_T^{s,r,q}} \leq C_T\left(||w_0||_{H^s \cap L^m} +  K||(w_d,w_e)||_{|W_T^{s,r,q}}\right)
\end{align*}
and therefore, choosing $||w_0||_{H^s\cap L^m}$ and $K$ small enough yields the claim. Similarly, one shows that $\mathcal{T}$ is a contraction using (\ref{c_bound_1}), (\ref{c_bound_2}), (\ref{c_bound_3}), (\ref{c_bound_4}), (\ref{c_bound_5}) and Lemma \ref{upgrade_lemma_w_e}. By the Banach fixed point theorem, this concludes the proposition.

\end{proof}
\bibliographystyle{abbrv}
\bibliography{refs1}
\end{document}